\documentclass[11pt,reqno,a4paper]{amsart}
\usepackage{fullpage}
\usepackage{amsmath}
\usepackage{lscape}
\usepackage{amscd}
\usepackage{amssymb}

\def\<#1,#2>{\langle\,#1,\,#2\,\rangle}

\def\kts{\mathcal{K}\mathcal{T}}
\def\akts{\mathcal{A}\mathcal{K}\mathcal{T}}
\def\hkts{\mathcal{H}\mathcal{K}\mathcal{T}}
\def\inf{\mathrm{inf}}
\def\dim{\mathrm{dim}}
\def\tm{\mathrm{TM}}
\def\tc{\mathrm{TM}^{\mathbb{C}}}
\def\tcip{{\mathrm{TM}^{(1,0)}}}
\def\tcim{{\mathrm{TM}^{(0,1)}}}
\def\sp{\mathrm{Spin}(2m)}
\def\lodd{\lambda_0^{odd}}
\def\leven{\lambda_0^{even}}
\def\s#1{\mathrm{\Sigma}_{#1} \mathrm{M}}
\def\sigm{{\mathrm{\Sigma M}}}
\def\of{\{e_j\}_{j=1,\ldots,n}}

\def\qed{\ensuremath{\hfill\Box}}
\newcommand{\ric}{\mathrm{Ric}}

\newtheorem{Lemma}{Lemma}[section]
\newtheorem{Proposition}[Lemma]{Proposition}
\newtheorem{Theorem}[Lemma]{Theorem}
\newtheorem{Corollary}[Lemma]{Corollary}
\newtheorem{Definition}[Lemma]{Definition}
\newtheorem{Remark}[Lemma]{Remark}
\newtheorem{Example}[Lemma]{Example}

\numberwithin{equation}{section}

\title{K\"ahlerian Twistor Spinors}
\author{Mihaela Pilca}
\thanks{The author thanks Graduiertenkolleg 1269 "Global Structures in Geometry and Analysis" for financial support and the Centre de Math{\'e}matiques ``Laurent Schwartz" de l'{\'E}cole Polytechnique for hospitality during part of the preparation of this work, within the French-German cooperation project Procope no. 17825PG}
\address{Mihaela Pilca \\ Mathematisches Institut\\ Universit\"at zu K\"oln\\ Weyertal 86-90 D-50931 K\"oln\\ Germany}
\email{mpilca@mi.uni-koeln.de}

\begin{document}

\begin{abstract}
On a K\"ahler spin manifold, K\"ahlerian twistor spinors are a natural analogue of twistor spinors on Riemannian spin manifolds. They are defined as 
sections in the kernel of a first order differential operator adapted to the K\"ahler structure, called K\"ahlerian twistor (Penrose) operator. We study K\"ahlerian twistor spinors and give a complete description of compact K{\"a}hler manifolds of constant scalar curvature admitting such spinors. As in the Riemannian case, the existence of K\"ahlerian twistor spinors is related to the lower bound of the spectrum of the Dirac operator.

\vspace{0.1cm}

keywords: Dirac operator, lower bound, K\"ahlerian twistor operator, K\"ahlerian twistor spinor. 
\end{abstract}

\maketitle

\tableofcontents

\section*{Introduction}

The purpose of this paper is to study the analogue of twistor spinors on K\"ahler spin manifolds and to describe the manifolds that admit such spinors. 

On a Riemannian spin manifold $(M,g)$, a special class of spinors exists, the so-called {\it twistor spinors}. They are defined as sections in the kernel of a natural first order operator, the {\it twistor operator}, which is given by the projection of the covariant derivative onto the Cartan summand of the tensor product $\mathrm{T^*M}\otimes\sigm$ (where $\mathrm{T^*M}$ is the cotangent bundle and $\sigm$ is the spinor bundle). More precisely, a twistor spinor $\varphi\in\Gamma(\sigm)$ is a solution of the equation
\begin{equation*}
\nabla_X\varphi=-\frac{1}{n}X\cdot D\varphi,
\end{equation*}
where $D$ is the Dirac operator. Twistor spinors are conformally invariant. In \cite{fr2}, Th.~Friedrich found their conformal relation to {\it Killing spinors}, which build up an important special class of twistor spinors. Killing spinors also play an important role in physics and are closely related to the spectrum of the Dirac operator as shown below.

The problem of finding optimal lower bounds for the eigenvalues of the Dirac operator on compact manifolds was first considered in 1980 by Th.~Friedrich, \cite{fr1}. He proved that on a compact spin manifold $(M^n,g)$ of positive scalar curvature $S$, the first eigenvalue $\lambda$ of $D$ satisfies
\[\lambda^2\geq\frac{n}{4(n-1)}\underset{M}{\inf} S.\]

The limiting case of this equality is characterized by the existence of real Killing spinors or equivalently by constant scalar curvature and the existence of twistor spinors. The general geometric description of simply connected manifolds carrying Killing spinors was obtained in 1993 by Ch.~B\"ar, \cite{baer}. 

As shown by O.~Hijazi in 1984, \cite{hijth}, K\"ahler spin manifolds cannot bear any nontrivial Killing spinors. Moreover, in 1992 K.-D.~Kirchberg proved, \cite{kirch1}, that if the scalar curvature is nonzero, then nontrivial twistor spinors cannot exist. It is thus natural to ask for an analogue class of spinors on K\"ahler manifolds, defined by a twistorial equation adapted to the K\"ahler structure. These spinors are called {\it K\"ahlerian twistor spinors} and are defined in the following way. On a K\"ahler spin manifold $(M^{2m},g,J)$, the spinor bundle $\sigm$ splits into $\mathrm{U}(m)$-irreducible subbundles: $\sigm=\oplus_{r=0}^{m}\s{r}$, where $\s{r}$ is the eigenbundle of the Clifford multiplication with the K\"ahler form for the eigenvalue $i(2r-m)$. For each $0\leq r\leq m$,  a K\"ahlerian twistor operator is defined by the projection of the covariant derivative onto the Cartan summand of the tensor product $\mathrm{T^*M}\otimes\s{r}$. The sections in the kernel of this first order differential operator are the K\"ahlerian twistor spinors. Explicitly, they satisfy the equations

\begin{equation*}
\begin{cases} 
\nabla_{X^{+}}\varphi=-\frac{1}{2(m-r+1)}X^{+}\cdot D^{-}\varphi,\\
\nabla_{X^{-}}\varphi=-\frac{1}{2(r+1)}X^{-}\cdot D^{+}\varphi,
\end{cases}
\end{equation*}
where $D^+$ and $D^-$ are defined by \eqref{defd+d-}. As in the Riemannian case, K\"ahlerian twistor spinors are closely related to the spectrum of the Dirac operator as shown below.

In 1986, Kirchberg improved Friedrich's inequality for K\"ahler manifolds. He showed, \cite{kirch86}, that every eigenvalue $\lambda$ of the Dirac operator on a compact K\"ahler manifold $(M^{2m},g,J)$ of positive scalar curvature $S$ satisfies 

\[\lambda^2\geq\frac{m+1}{4m}\underset{M}{\inf} S, \quad \text{if $m$ is odd,}\]

and 

\[\lambda^2\geq\frac{m}{4(m-1)}\underset{M}{\inf} S,  \quad \text{if $m$ is even.}\]

The manifolds which satisfy the limiting case of these inequalities are characterized by the existence of {\it K\"ahlerian Killing spinors} (see \eqref{ecodd}) for $m$ odd and by a spinor satisfying a similar equation (see \eqref{eceven}) for $m$ even. The limiting manifolds were geometrically described by A.~Moroianu in 1994 for odd complex dimension, respectively in 1999 for even complex dimension, \cite{am_odd}, \cite{am_even}. 

We note that the spinors characterizing the limiting manifolds of Kirchberg's inequa\-li\-ties, \emph{i.e.} those satisfying the equations \eqref{ecodd} and \eqref{eceven}, are in particular K\"ahlerian twistor spinors in $\s{\frac{m\pm1}{2}}$, respectively $\s{\frac{m}{2}\pm 1}$. It is thus natural to study K\"ahlerian twistor spinors as a generalization of these two important special cases.

The first eigenvalue $\lambda^2$ of the square of the Dirac operator restricted to $\s{r}$ on a compact K\"ahler manifold $(M^{2m},g,J)$ of positive scalar curvature $S$ satisfies the following inequality for $0\leq r\leq \frac{m}{2}$ (for $\frac{m}{2}< r\leq m$ there is a similar inequality, see \eqref{valmin2}):

\[\lambda^2\geq\frac{2(r+1)}{4(2r+1)}\underset{M}{\inf} S.\]

The limiting manifolds are characterized by constant scalar curvature and the existence of nontrivial K\"ahlerian twistor spinors in $\s{r}$ (see Proposition~\ref{eigentyper}). In fact these inequalities on each $\s{r}$ provide a proof of Kirchberg's inequalities (\emph{cf.} \cite{kirch2}). 

The main result of this paper is the geometric description of the limiting manifolds of these inequalities on $\s{r}$ ($0<r<m$), \emph{i.e.} of spin K\"ahler manifolds of constant scalar curvature carrying nontrivial K\"ahlerian twistor spinors in $\s{r}$. More precisely, we obtain the following (see Theorems~\ref{classke} and~\ref{class}):

\begin{Theorem}
Let $(M^{2m},g,J)$ be a compact simply connected spin K\"ahler manifold of constant scalar curvature admitting nontrivial non-parallel K\"ahlerian twistor spinors in $\s{r}$ for an $r$ with $0<r<m$. Then $M$ is the product of a Ricci-flat manifold $M_1$ and an irreducible K\"ahler-Einstein manifold $M_2$, which is a limiting manifold for Kirchberg's inequality in odd complex dimensions and thus is either the complex projective space in complex dimension $4k+1$ or, in complex dimension $4k+3$, a twistor space over a quaternionic K\"ahler manifold of positive scalar curvature. More precisely, there exist anti-holomorphic (holomorphic) K\"ahlerian twistor spinors in at most one such $\s{r}$ with $r<\frac{m}{2}$ ($r>\frac{m}{2}$) and they are of the form:
\begin{equation*}
\psi=\xi_0\otimes\varphi_r \quad (\psi=\xi_{2r-m-1}\otimes\varphi_{m-r+1}),
\end{equation*}
where $\xi_0\in\Gamma(\mathrm{\Sigma}_0 \mathrm{M}_1)$ ($\xi_{2r-m-1}\in\Gamma(\mathrm{\Sigma}_{2r-m-1} \mathrm{M}_1)$) is a parallel spinor and  $\varphi_r\in\Gamma(\mathrm{\Sigma}_{r} \mathrm{M}_2)$ ($\varphi_{m-r+1}\in\Gamma(\mathrm{\Sigma}_{m-r+1} \mathrm{M}_2)$) is an anti-holomorphic (holomorphic) K\"ahlerian twistor spinor. In particular, the complex dimension of the K\"ahler-Einstein manifold $M_2$ is $2r+1$ (resp. $2(m-r)+1$).
\end{Theorem}

For $r=\frac{m}{2}\pm 1$ the complex dimension of the Ricci-flat factor is $1$ and we reobtain the limiting manifolds of Kirchberg's inequalities for $m$ even. Thus, our result may be considered on the one hand as a generalization of A.~Moroianu's description of limiting K\"ahler manifolds in even complex dimension, while on the other hand, we use his classification in the odd-dimensional case. In particular, our result answers a question raised by K.-D.~Kirchberg in \cite{kirch2} and, in a certain sense, completes the picture in the K\"ahler case.

We note that a slightly different notion of 'K\"ahlerian twistor spinors' has already been introduced by K.-D.~Kirchberg, \cite{kirch2}, and O.~Hijazi, \cite{hij}. This is a special class of K\"ahlerian twistor spinors (defined as sections in the kernel of the K\"ahlerian twistor operator), which we call {\it special K\"ahlerian twistor spinors} and are characterized by a further condition, \emph{i.e.} to be in the kernel of $D^-$ or $D^+$ (see Remark~\ref{relation}).

The paper is organized as follows. After a short preliminary section, where we introduce the notation and some results of spin geometry on K\"ahler manifolds, we define in \S ~\ref{sectkts} the main objects, the K\"ahlerian twistor spinors, and study some particular cases. In \S ~\ref{sectktc} we construct a connection called {\it K\"ahlerian twistor connection}, such that K\"ahlerian twistor spinors are in one-to-one correspondence with parallel sections of this connection. Furthermore, in \S ~\ref{sectcurv} we compute the curvature of the K\"ahlerian twistor connection, which allows us to derive some useful formulas. These are the starting points for the main part of the paper, \S ~\ref{classif}, where we prove the above mentioned result (Theorem~\ref{class}). We first show that on a compact K\"ahler spin manifold of constant scalar curvature all K\"ahlerian twistor spinors are special K\"ahlerian twistor spinors; then we show that the existence of such a nontrivial spinor imposes strong restrictions on the Ricci tensor, namely it only has two constant eigenvalues. This has already been proven by A.~Moroianu, \cite{am_even1}, in the special case of limiting manifolds of Kirchberg's inequality for even complex dimension and we notice that his method works for any bundle $\s{r}$. By a result of V.~Apostolov, T.~Dr\u aghici and A.~Moroianu, \cite{adm}, we derive that the Ricci tensor must be parallel. Thus, assuming that the manifold is simply connected, it must be, by de~Rham's decomposition theorem, a product of irreducible K\"ahler-Einstein manifolds. Analyzing K\"ahlerian twistor spinors on a product (Theorem~\ref{product}), it turns out that one of the factors is Ricci-flat and the other is K\"ahler-Einstein admitting itself special K\"ahlerian twistor spinors. The problem is thus reduced to the study of K\"ahler-Einstein manifolds, where we show that the only nontrivi\-al non-extremal K\"ahlerian twistor spinors are the K\"ahlerian Killing spinors (Proposition~\ref{ke}). In the last section we consider a larger class of manifolds, the weakly Bochner flat manifolds, and show that the existence of K\"ahlerian twistor spinors already implies constancy of the scalar curvature. 

\smallskip

\noindent \textsc{Acknowledgments.} This paper is a part of my Ph.D. thesis. I thank Uwe Semmelmann, my supervisor, for his encouragement. I am grateful to Andrei Moroianu for many valuable discussions and suggestions. I also thank K.-D. Kirchberg for reading the preliminary version and suggesting me a better terminology for the special classes of K\"ahlerian twistor spinors.

\vspace{0.1 cm}

\section{Preliminaries: Spin Geometry on K\"ahler Manifolds}\label{sectprelim}

\subsection{The Decomposition of the Spinor Bundle on K\"ahler Manifolds}

Let $(M,g,J)$ be a K\"ahler manifold of real dimension $n=2m$ with Riemannian metric $g$, complex structure $J$ and K\"ahler form $\Omega=g(J\cdot, \cdot)$. The tangent and cotangent bundle are identified using the metric $g$. Where we do not write sums, we implicitly use the Einstein summation convention over repeated indices. $\of$ always denotes a local orthonormal frame. The complexified tangent bundle splits into the $\pm i$-eigenbundles of the complex structure: $\tc=\tcip\oplus\tcim$ and we denote the components of a vector field $X$ with respect to this splitting as follows:
\[X^+=\frac{1}{2}(X-iJX)\in\Gamma(\tcip),\quad X^-=\frac{1}{2}(X+iJX)\in\Gamma(\tcim).\]

We now assume on $M$ the existence of a spin structure. In the case of K\"ahler manifolds this is equivalent to the existence of a square root of the canonical bundle $K=\Lambda^{(m,0)}M$, \emph{i.e.} a holomorphic line bundle $L$ such that $K\cong L\otimes L$ (see \cite{hitchin}). 

Let $P_{\mathrm{Spin}}$ be the $\sp$-principal bundle of the spin structure and denote by $\sigm$ the associated spinor bundle: $\sigm=P_{\mathrm{Spin}}\times_{\sp}\Sigma$, where $\Sigma$ is the $2^m$-dimensional complex spin representation of $\sp$. $\sigm$ is a complex Hermitian vector bundle and its sections are called {\it spinor fields} (or shortly {\it spinors}).

The Clifford contraction $c:\mathrm{T^*M}\otimes\sigm \to \sigm$ is defined on each fiber by the Clifford multiplication on the spinor representation $\Sigma$. On decomposable elements we have $c(X\otimes \varphi)=X\cdot\varphi$. It is extended to a multiplication with $k$-forms. Each $k$-form $\alpha$ acts as an endomorphism of the spinor bundle, which is locally given by
\[\alpha\cdot\varphi=\sum_{1\leq i_1 < i_2 < \dots < i_k\leq 2m} \alpha(e_{i_1}, \dots ,e_{i_k})e_{i_1}\cdot \ldots \cdot e_{i_k}\cdot\varphi.\] 

Consider now the Clifford multiplication with the complex volume form (with the orientation given by the complex structure): $\omega^\mathbb{C}=i^{m}\prod_{i=1}^{m}e_i\cdot Je_i$. Since the dimension is even, $\omega^\mathbb{C}$ has the eigenvalues of $+1$ and $-1$ and the corresponding eigenspaces
\begin{equation}\label{dec+-}
\Sigma=\Sigma^+\oplus \Sigma^-
\end{equation}
are inequivalent complex irreducible representations of $\sp$.

As an endomorphism of the spinor bundle, the K\"ahler form is given locally by
\begin{equation}\label{omegacliff}
\Omega=\frac{1}{2}\sum_{j=1}^{n}e_j\cdot J e_j.
\end{equation}

By a straightforward computation it follows:
\begin{Lemma}\label{decompsp}
Under the action of the K\"ahler form $\Omega$, the spinor bundle splits into the orthogonal sum of holomorphic subbundles
\begin{equation}\label{decomp}
\sigm=\oplus_{r=0}^{m}\s{r},
\end{equation}
where each $\s{r}$ is an eigenbundle of $\Omega$ corresponding to the eigenvalue $i\mu_r=i(2r-m)$ and  $\mathrm{rank}_{\mathbb{C}}(\s{r})=\binom{m}{r}$.
\end{Lemma}

This decomposition corresponds to the one for $(0,*)$-forms in $(0,r)$-forms on $M$, if we consider the so-called {\it Hitchin representation} (\cite{hitchin}) of the spin bundle of any almost Hermitian manifold: $\sigm\cong L\otimes\Lambda^{0,*}$, where $L$ is the square root of the canonical bundle $K$ of $M$ determined by the spin structure.

Comparing the decomposition \eqref{dec+-} with the finer one \eqref{decomp} we have
\[\Sigma^+ M=\underset{\begin{subarray}{c}
       0\leq r\leq m\\ \text{$r$ even}
      \end{subarray}}{\oplus} \s{r}, \quad  \Sigma^- M=\underset{\begin{subarray}{c}
       0\leq r\leq m\\ \text{$r$ odd}
      \end{subarray}}{\oplus} \s{r}.\]
On the spinor bundle there is a canonical $\mathbb{C}$-anti-linear real or quaternionic structure \mbox{$\mathfrak{j}:\sigm \to \sigm$} such that $\mathfrak{j}^2=(-1)^{\frac{m(m+1)}{2}}$ and with the following property:
\[\mathfrak{j}:\s{r} \to \s{m-r}, \quad \mathfrak{j}(Z\cdot\varphi)=\bar Z\cdot\mathfrak{j}(\varphi), \quad \text{for } Z\in\Gamma(\tc).\]

\vspace{0.1 cm}

\subsection{The Dirac Operator and Estimates for Its Eigenvalues}

The Levi-Civita con\-nection $\nabla$ on $\tm$ induces a covariant derivative on $\sigm$, which we also denote by $\nabla$. Since the K\"ahler form is parallel, $\nabla$ preserves the splitting \eqref{decomp}.

The {\it Dirac operator} is defined as the composition 
\begin{equation}\label{dirac}
\Gamma(\sigm) \overset{\nabla}{\rightarrow} \Gamma(\mathrm{T^*M}\otimes\sigm) \overset{c}{\rightarrow} \Gamma(\sigm), \quad D=c \circ\nabla.
\end{equation}
Explicitly $D$ is locally given by
\begin{equation}\label{diracloc}
D =\sum_{j=1}^{n}e_j\cdot \nabla_{e_j}.
\end{equation}

Associated with the complex structure $J$ there is another ``square root of the Laplacian", locally defined by
\[D^c =\sum_{j=1}^{n} Je_j\cdot \nabla_{e_j}.\]
$D^c$ is also an elliptic self-adjoint operator and it follows easily that
\[(D^c)^2=D^2\hspace{0.3 cm} \text{and}\hspace{0.3 cm}  DD^c+D^cD=0.\] 

Define now the two operators
\begin{equation}\label{defd+d-}
D^+ =\frac{1}{2}(D-iD^c)=\sum_{j=1}^{n}e_j^+\cdot \nabla_{e_j^-}\quad \quad D^- =\frac{1}{2}(D+iD^c)=\sum_{j=1}^{n}e_j^-\cdot \nabla_{e_j^+},
\end{equation}
which satisfy the relations
\begin{equation} \label{d2sum}
D=D^+ +D^-, \quad (D^+)^2=0, \quad (D^-)^2=0, \quad D^+D^- + D^-D^+ =D^2.
\end{equation}

When restricting the Dirac operator to $\s{r}$, it acts as
\[D=D^+ +D^-:\Gamma(\s{r}) \to \Gamma(\s{r-1}\oplus\s{r+1}),\]
because of the following result which can be checked by straightforward computation.

\begin{Lemma}\label{cr}
For any tangent vector field $X$ and $r\in\{0,\ldots,m\}$ one has
\begin{equation}\label{cliffmult}
X^+\cdot\s{r}\subset \s{r+1} \quad X^-\cdot\s{r}\subset \s{r-1},
\end{equation}
with the convention that $\s{-1}=\s{m+1}=M\times\{0\}$. Thus, if we denote by $c_r$ the restriction of the Clifford contraction to $\mathrm{T^*M}\otimes\s{r}$, then $c_r$ splits as follows
\[c_r=c_r^+\oplus c_r^-:\Gamma(\tm\otimes \s{r})\to \Gamma(\s{r-1})\oplus\Gamma(\s{r+1}).\]
\end{Lemma}

One of the main tools for the study of the Dirac operator is the Schr{\"o}dinger-Lichnerowicz formula:
\begin{equation}\label{lichn}
D^2=\nabla^*\nabla+\frac{1}{4}S,
\end{equation}
\noindent where $\nabla^*\nabla$ is the Laplacian on the spinor bundle and $S$ is the scalar curvature of $M$.

Let us recall here for later use the lower bounds for the spectrum of the Dirac ope\-rator on Riemannian and K\"ahler manifolds. The first such inequality was obtained by Th.~Friedrich, \cite{fr1}, who showed that on an $n$-dimensional compact Riemannian spin mani\-fold $(M,g)$ each eigenvalue $\lambda$ of the Dirac operator satisfies 
\begin{equation}\label{fried}
\lambda^2\geq\frac{n}{4(n-1)}\underset{M}{\inf} S.
\end{equation}

Of course, this inequality gives new information only if the scalar curvature is positive, in which case we denote the smallest possible eigenvalue by $\lambda_0=\sqrt{\frac{n}{4(n-1)}\underset{M}{\inf} S}$. In \cite{fr1} it is shown that a limiting manifold for \eqref{fried} is characterized by the existence of a special spinor. More precisely, we have

\begin{Theorem}\label{charactriem}
Let $(M,g)$ be a Riemannian manifold which admits an eigenspinor $\varphi$ of the Dirac operator $D$ with the smallest eigenvalue $\lambda_0$. Then the manifold is Einstein and $\varphi$ is a Killing spinor for the Killing constant $-\frac{\lambda_0}{n}$, \emph{i.e.} satisfies the equation
\begin{equation}\label{kill}
\nabla_X\varphi=-\frac{\lambda_0}{n} X\cdot\varphi,
\end{equation}
for all vector fields $X$ on $M$. Conversely, if $\varphi$ is a nontrivial spinor on $M$ satisfying equation \eqref{kill} for some real constant, then $g$ is an Einstein metric and $(M,g)$ is a limiting manifold for \eqref{fried}, $\varphi$ being an eigenspinor of $D$ corresponding to the smallest eigenvalue $\lambda_0$.
\end{Theorem}

The complete simply connected Riemannian manifolds carrying real Killing spinors have been described by Ch.~B\"ar \cite{baer}. The main tool in his proof is the cone construction. He shows that Killing spinors correspond to fixed points of the holonomy group of the cone and then uses the Berger-Simons classification of possible holonomy groups. 

On K\"ahler manifolds the inequality \eqref{fried} is always strict since a K\"ahler manifold does not admit Killing spinors. It was improved by K.-D.~Kirchberg, who showed that each eigenvalue $\lambda$ of the Dirac operator on a $2m$-dimensional compact spin K\"ahler manifold $(M,g,J)$ satisfies 
\begin{equation}\label{kirchodd}
\lambda^2\geq\frac{m+1}{4m}\underset{M}{\inf} S, \quad \text{if $m$ is odd,}
\end{equation}
\begin{equation}\label{kircheven}
\quad\lambda^2\geq\frac{m}{4(m-1)}\underset{M}{\inf} S,  \quad \text{if $m$ is even.}
\end{equation}

Again we can only get new information about the eigenvalues from these inequalities if the scalar curvature is positive. In this case we denote the smallest possible eigenvalue by $\lodd:=\sqrt{\frac{m+1}{4m}\underset{M}{\inf} S}$ and $\leven:=\sqrt{\frac{m}{4(m-1)}\underset{M}{\inf} S}$. The limiting manifolds of Kirchberg's inequalities are also characterized by the existence of spinors satisfying a certain differen\-tial equation. More precisely, K.-D.~Kirchberg \cite{kirch} and O.~Hijazi \cite{hij} proved:

\begin{Theorem}\label{charactodd}
Let $(M,g,J)$ be a compact K\"ahler spin manifold of complex dimension $m=2l+1$ which admits an eigenspinor $\varphi$ of $D$ corresponding to the smallest eigenvalue $\lodd$. Then the metric $g$ is Einstein and $\varphi=\varphi_{l}+\varphi_{l+1}\in\Gamma(\s{l}\oplus\s{l+1})$ is a K\"ahlerian Killing spinor with the Killing constant $-\frac{\lodd}{m+1}$, \emph{i.e.} its components satisfy the equations
\begin{equation}\label{ecodd}
\begin{split}
\nabla_X \varphi_l&= -\frac{\lodd}{m+1} X^-\cdot \varphi_{l+1},\\ 
\nabla_X \varphi_{l+1}&= -\frac{\lodd}{m+1} X^+\cdot \varphi_{l},
\end{split}
\end{equation}
for any vector field $X$. Conversely, if $\varphi=\varphi_{l}+\varphi_{l+1}\in\Gamma(\s{l}\oplus\s{l+1})$ is a spinor on $M$ satisfying the equations \eqref{ecodd} for some real constant, then $g$ is an Einstein metric and $(M,g,J)$ is a limiting manifold for \eqref{kirchodd}, $\varphi$ being an eigenspinor of $D$ corresponding to the smallest eigenvalue $\lodd$.
\end{Theorem}

In the case of even complex dimension, the characterization of limiting manifolds in terms of special spinors has been given by K.-D. Kichberg \cite{kirch2} and in the following stronger version by P.~Gauduchon \cite{pg}.

\begin{Theorem}\label{characteven}
Let $(M,g,J)$ be a compact K\"ahler spin manifold of complex dimension $m=2l\geq 4$ which admits an eigenspinor $\varphi$ of $D$ corresponding to the smallest eigenvalue $\leven$. Then the manifold has constant scalar curvature and $\varphi=\varphi_{l-1}+\mathfrak{j}(\varphi'_{l-1})$, where $\varphi_{l-1},\varphi'_{l-1}\in\Gamma(\s{l-1})$ are spinors satisfying the following equation:
\begin{equation}\label{eceven}
\nabla_X \varphi_{l-1}= -\frac{1}{m}X^-\cdot D\varphi_{l-1},
\end{equation}
for any vector field $X$. Conversely, if $\varphi=\varphi_{l-1}+\mathfrak{j}(\varphi'_{l-1})\in\Gamma(\s{l-1}\oplus\s{l+1})$ is a spinor on $M$ such that $\varphi_{l-1},\varphi'_{l-1}$ satisfy the equation \eqref{eceven} and if there exists a real nonzero constant $\lambda$ such that $D^2\varphi=\lambda^2\varphi$, then the manifold has constant scalar curvature and is a limiting manifold for \eqref{kircheven}, $\lambda$ being equal to $\leven$.
\end{Theorem}

These limiting manifolds have been classified by A.~Moroianu in \cite{am_odd} for $m$ odd and in \cite{am_even} for $m$ even. In the even complex dimension, the result was conjectured by A. Lichnerowicz \cite{lich}, who proved it under the assumption that the Ricci tensor is parallel.

\begin{Theorem}\label{classodd}
The only limiting manifold for \eqref{kirchodd} in complex dimension $4k+1$ is the complex projective space $\mathbb{C}P^{4k+1}$. In complex dimension $4k+3$ the limiting manifolds are exactly the twistor spaces over quaternionic K\"ahler manifolds of positive scalar curvature.
\end{Theorem}

\begin{Theorem}\label{classeven}
A K\"ahler manifold $M$ of even complex dimension $m\geq 4$ is a limiting manifold for \eqref{kircheven} if and only if its universal cover is isometric to a Riemannian product $N\times\mathbb{R}^2$, where $N$ is a limiting manifold for the odd complex dimension $m-1$ and $M$ is the suspension over a flat parallelogram of two commuting isometries of $N$ preserving a K\"ahlerian Killing spinor.
\end{Theorem}

\section{K\"ahlerian Twistor Spinors}\label{sectkts}

\subsection{Twistor Operators}

Natural first order differential operators acting on sections of an associated vector bundle $E$ over a manifold $M$ are given by the composition of projections onto irreducible components of the tensor product $\mathrm{T^*M}\otimes E$ with a covariant derivative on $E$. Then the principal symbol of the operator is the projection defining it. There is always a distinguished projection onto the so-called {\it Cartan summand}, whose highest weight is exactly the sum of the highest weights of the representations defining the bundles $\mathrm{T^*M}$ and $E$.

Consider now the spinor bundle $\sigm$ over a Riemannian spin manifold $(M,g)$. As the tensor product $\tm\otimes\sigm$ splits as $Spin(n)$-representation as follows: 
\[\tm\otimes\sigm\cong \sigm\oplus\ker(c),\]
we get two first order differential operators: the Dirac operator \eqref{dirac} is given by the projection (which is identified with the Clifford contraction) of the covariant derivative onto $\sigm$ and the complementary operator given by the projection of the covariant derivative onto the Cartan summand $\ker(c)$. 

In order to define the projections we need to consider an embedding of $\sigm$ into the tensor product $\tm\otimes\sigm$, \emph{e.g.} the right inverse of $c$, $\iota: \sigm\to \tm\otimes\sigm$ such that $c\circ\iota=\mathrm{id}_{\sigm}$, which is given locally as follows:
\[\iota(\varphi)=-\frac{1}{n}\sum_{j=1}^{n}e_j\otimes e_j\cdot\varphi.\]
The {\it Riemannian twistor (Penrose) operator} is then defined by
\[T: \Gamma(\sigm)\to \Gamma(\ker c), \quad T\varphi=\nabla\varphi+\frac{1}{n}e_j\otimes e_j\cdot D\varphi,\]
or, more explicitly, when applied to a vector field $X$:
\begin{equation}\label{riemtwist}
T_X\varphi=\nabla_X \varphi+\frac{1}{n}X\cdot D\varphi,
\end{equation}

\begin{Definition}\label{defrts}
Let $(M,g)$ be a Riemannian manifold. A spinor $\varphi\in\Gamma(\sigm)$ is called {\em Riemannian twistor spinor} if it belongs to the kernel of the Riemannian twistor operator, \emph{i.e.} if it satisfies the differential equation
\begin{equation}\label{ecriemtwist}
\nabla_X \varphi=-\frac{1}{n}X\cdot D\varphi,
\end{equation}
for all vector fields $X$.
\end{Definition}

Let us now consider the case of K\"ahler manifolds. It was proven by K.-D.~Kirchberg \cite{kirch1} that on a K\"ahler manifold with nonzero scalar curvature, the space of Riemannian twistor spinors is trivial (a different proof of this result is also given by O.~Hijazi \cite{hij} and for compact K\"ahler manifolds this vanishing result is due to Lichnerowicz \cite{lichn1}). Thus it is natural to consider a twistor operator adapted to the K\"ahler structure, by looking at the decomposition of each tensor product of the vector bundles $\tm\otimes \s{r}$ (for $r=0,\ldots, m$) into irreducible components under the action of the unitary group $\mathrm{U}(m)$. There are three irreducible summands:
\begin{equation}\label{srdecomp}
\tm\otimes\s{r}\cong \s{r-1}\oplus \s{r+1}\oplus \ker(c_r),
\end{equation}
where $c_r$ is the restriction of the Clifford contraction to $\s{r}$ as in Lemma~\ref{cr}. Thus, there are three first order differential operators: the first two projections are given by $c_r^-$ and $c_r^+$ respectively and the third one is the projection onto the Cartan summand, $\ker c_r$. As in the Riemannian case we need the two embeddings, which are locally given as follows:
\[\iota_r^-: \s{r-1}\to\tm\otimes \s{r}, \quad \iota_r^-(\varphi)=-\frac{1}{2(m-r+1)}\sum_{j=1}^{n}e_j\otimes e_j^+\cdot\varphi,\]
\[\iota_r^+: \s{r+1}\to\tm\otimes \s{r}, \quad \iota_r^+(\varphi)=-\frac{1}{2(r+1)}\sum_{j=1}^{n}e_j\otimes e_j^-\cdot\varphi.\]

The {\it K\"ahlerian twistor (Penrose) operator of type $r$} is then defined by the projection of the covariant derivative onto the Cartan summand:  
\[T_r:  \Gamma(\s{r})\to  \Gamma(\ker (c_r)), \quad T_r\varphi=\nabla\varphi+\frac{1}{2(m-r+1)}e_j\otimes e_j^+\cdot D^-\varphi +\frac{1}{2(r+1)}e_j\otimes e_j^-\cdot D^+\varphi,\]
or, more explicitly, when applied to a vector field $X$: 
\begin{equation}\label{kaehltwist}
(T_r)_{X}\varphi=\nabla_X \varphi+\frac{1}{2(m-r+1)}X^+\cdot D^-\varphi +\frac{1}{2(r+1)}X^-\cdot D^+\varphi.
\end{equation}

The K\"ahlerian twistor operator has already been introduced, \emph{e.g.} by P.~ Gauduchon \cite{pg}. Different approaches have been considered by O.~Hijazi \cite{hij} and by K.-D.~Kirchberg \cite{kirch2}. In Remark~\ref{relation} we discuss the relationship between the various definitions of a twistor spinor adapted to the K\"ahler structure.

\begin{Definition}\label{defkts}
Let $(M,g,J)$ be a K\"ahler manifold. A spinor $\varphi\in\Gamma(\s{r})$ is called {\em K\"ahlerian twistor spinor} if it belongs to the kernel of the K\"ahlerian twistor operator, \emph{i.e.} if it satisfies the differential equations
\begin{equation}\label{defec}
\begin{cases} 
\nabla_{X^{+}}\varphi=-\frac{1}{2(m-r+1)}X^{+}\cdot D^{-}\varphi,\\
\nabla_{X^{-}}\varphi=-\frac{1}{2(r+1)}X^{-}\cdot D^{+}\varphi,
\end{cases}
\end{equation}
for all vector fields $X$. 
\end{Definition}

\noindent We shall denote by $\kts(r)$ the space of K\"ahlerian twistor spinors in $\s{r}$. 

\noindent It follows immediately from the defining equations \eqref{defec} that the real or quaternionic structure $\mathfrak{j}$ of the spinor bundle preserves the space of K\"ahlerian twistor spinors: 
\[\mathfrak{j}:\kts(r) \stackrel{\sim}{\to} \kts(m-r).\]
Thus it is sufficient to study $\kts(r)$ for $0\leq r\leq \frac{m}{2}$.

An important class of spinors in $\kts(r)$ are the special K\"ahlerian twistor spinors defined as follows:

\begin{Definition}\label{defsemi}
A K\"ahlerian twistor spinor $\varphi\in\Gamma(\s{r})$ is {\em holomorphic}, respectively {\em anti-holomorphic K\"ahlerian twistor spinor} if $D^+\varphi=0$, respectively $D^-\varphi=0$.
\end{Definition}

We denote the space of holomorphic and anti-holomorphic K\"ahlerian twistor spinors in $\s{r}$ by $\hkts(r)$ and $\akts(r)$ respectively and notice that they are interchanged by $\mathfrak{j}$: 
\begin{equation}\label{strj}
\mathfrak{j}:\akts(r) \overset{\sim}{\to} \hkts(m-r).
\end{equation}

Parallel spinors are of course the easiest examples of a twistor spinor of any kind, because, by definition, all components of the covariant derivative vanish. However this condition is very restrictive, parallel spinors only exist on Ricci-flat K\"ahler manifolds.

Directly from the decomposition \eqref{srdecomp} and using the embeddings $\iota_r^+$ and $\iota_r^-$ we get the following Weitzenb\"ock formula, relating the differential operators acting on sections of $\s{r}$:
\begin{equation}\label{wbform}
\nabla^*\nabla=\frac{1}{2(r+1)}D^-D^+ +\frac{1}{2(m-r+1)}D^+D^- + T_r^*T_r.
\end{equation}

This is just a special case of a general Weitzenb\"ock formula which expresses the rough Laplacian $\nabla^*\nabla$ acting on an associated vector bundle $E$ as the sum of all $T^*T$ with $T$ first order differential operator given by the projections of a covariant derivative onto the irreducible components of the tensor product $\tm\otimes E$.

\begin{Remark}[Relationship to the estimates of the eigenvalues of the Dirac operator]
\normalfont{
By Theorem~\ref{charactriem}, each eigenspinor corresponding to the smallest eigenvalue of the Dirac operator on a Riemannian manifold is a Killing spinor, thus in particular a twistor spinor. Moreover, the eigenspinors corresponding to the smallest eigenvalue are exactly the eigenspinors of $D$ which are twistor spinors.

Similarly, on a K\"ahler manifold Theorems \ref{charactodd} and \ref{characteven} imply that every eigenspinor of the Dirac operator corresponding to the smallest eigenvalue is a sum of two special K\"ahlerian twistor spinors: if $m$ is odd, then \eqref{ecodd} implies that $\varphi\in\akts(\frac{m-1}{2})\oplus\hkts(\frac{m+1}{2})$ and if $m$ is even, then \eqref{eceven} implies that $\varphi\in\akts(\frac{m}{2}-1)\oplus\hkts(\frac{m}{2}+1)$. Moreover, eigenspinors corresponding to the smallest eigenvalue are exactly the eigenspinors of $D$ which are K\"ahlerian twistor spinors. The geometric  description of the limiting K\"ahler manifolds (Theorems \ref{classodd} and \ref{classeven}) provides the first examples of manifolds admitting K\"ahlerian twistor spinors. Thus, K\"ahlerian twistor spinors may be seen as a generalization of these special spinors which naturally appear in the limiting case for the lower bound of the spectrum of the Dirac operator. }
\end{Remark}

We now show that special K\"ahlerian twistor spinors on a K\"ahler manifold of positive scalar curvature are exactly the eigenspinors of the smallest eigenvalue of the square of the Dirac operator restricted to an irreducible subbundle $\s{r}$. This result has been proven by O.~Hijazi \cite{hij} and by K.-D.~Kirchberg \cite{kirch2}. Here we follow the argument given by P. Gauduchon \cite{pg} and by U. Semmelmann \cite{uweshort}.

\begin{Lemma}\label{inegnorme}
Let $\varphi\in\Gamma(\s{r})$. Then the following inequality holds
\begin{equation}\label{ineg}
|\nabla\varphi|^2\geq \frac{1}{2(r+1)}|D^+\varphi|^2 +\frac{1}{2(m-r+1)}|D^-\varphi|^2,
\end{equation}
with equality if and only if $\varphi$ is a K\"ahlerian twistor spinor, \emph{i.e.} $T_r\varphi=0$.
\end{Lemma}

\begin{proof}
The statement of the lemma is a direct consequence of the following relation:
\[|\nabla\varphi|^2=\frac{1}{2(r+1)}|D^+\varphi|^2+\frac{1}{2(m-r+1)}|D^-\varphi|^2+|T_r\varphi|^2,\]
which in turn is implied by the following equalities that are straightforward from the definition of the embeddings $\iota^{\pm}_r$:
\[|\iota^+_r(\varphi_{r+1})|^2=\frac{1}{2(r+1)}|\varphi_{r+1}|^2,\]
\[|\iota^-_r(\varphi_{r-1})|^2=\frac{1}{2(m-r+1)}|\varphi_{r-1}|^2,\]
\[\nabla\varphi=\iota^+_r(D^+\varphi)+\iota^-_r(D^-\varphi)+T_r\varphi.\qed\]
\end{proof}

The Lichnerowicz formula \eqref{lichn} yields the following:
\begin{Lemma}\label{alg}
Let $(M,g)$ be a compact spin manifold. If $\varphi$ is an eigenspinor of $D^2$, $D^2\varphi=\lambda\varphi$ and satisfies the inequality
\begin{equation}\label{inegalg}
|\nabla\varphi|^2\geq\frac{1}{k}|D\varphi|^2,
\end{equation}
then
\[\lambda\geq\frac{k}{k-1}\frac{1}{4}\underset{M}{\inf} S\]
and equality is attained if and only if $S$ is constant and equality in \eqref{inegalg} holds at all points of the manifold.
\end{Lemma}

\begin{Proposition}\label{eigentyper}
Let $(M,g,J)$ be a compact K\"ahler manifold of positive scalar curvature. Then any eigenvalue $\lambda$ of $D^2$ on $\s{r}$ satisfies:
\begin{equation}\label{valmin1}
\lambda\geq\frac{2(r+1)}{2r+1}\frac{1}{4}\underset{M}{\inf} S, \quad \text{if $r\leq\frac{m}{2}$}
\end{equation}
and
\begin{equation}\label{valmin2}
\lambda\geq\frac{2(m-r+1)}{2m-2r+1}\frac{1}{4}\underset{M}{\inf} S, \quad \text{if $r>\frac{m}{2}$}.
\end{equation}
Equality is attained if and only if the scalar curvature is constant and the corresponding eigenspinor is an anti-holomorphic (holomorphic) K\"ahlerian twistor spinor if  $r<\frac{m}{2}$ ($r>\frac{m}{2}$).
\end{Proposition}

\begin{proof}
Let $\varphi\in\Gamma(\s{r})$ with $D^2\varphi=\lambda\varphi$. We distinguish two cases.\\

I. If $D^-\varphi=0$, then $|D\varphi|^2=|D^+\varphi|^2$ and \eqref{ineg} implies:
\[|\nabla\varphi|^2\geq \frac{1}{2(r+1)}|D^+\varphi|^2=\frac{1}{2(r+1)}|D\varphi|^2.\]
Applying Lemma \ref{alg}, it follows that 
\[\lambda\geq\frac{2(r+1)}{2r+1}\frac{1}{4}\underset{M}{\inf} S.\]

II. If $D^-\varphi\neq 0$, then we apply the same argument for $\varphi^-:=D^-\varphi\in\Gamma(\s{r-1})$ with $D^2\varphi^-=\lambda\varphi^-$. Then $D^-\varphi^-=0$, so that $|D\varphi^-|^2=|D^+\varphi^-|^2$ and from \eqref{ineg} it follows:
\[|\nabla\varphi^-|^2\geq \frac{1}{2r}|D^+\varphi^-|^2=\frac{1}{2r}|D\varphi^-|^2.\]
Applying again Lemma~\ref{alg}, it follows that
\[\lambda\geq\frac{2r}{2r-1}\frac{1}{4}\underset{M}{\inf} S>\frac{2(r+1)}{2r+1}\frac{1}{4}\underset{M}{\inf} S.\]

The same argument applied to the cases when $D^+\varphi=0$ and $D^+\varphi\neq 0$ shows that $\lambda\geq\frac{2(m-r+1)}{2m-2r+1}\frac{1}{4} \underset{M}{\inf}S$. If $r\leq\frac{m}{2}$, then $\frac{2(m-r+1)}{2m-2r+1}\frac{1}{4}\underset{M}{\inf}S\leq\frac{2(r+1)}{2r+1}\frac{1}{4}\underset{M}{\inf}S$ and thus follows \eqref{valmin1}, otherwise \eqref{valmin2}. The equality case follows from Lemmas \ref{inegnorme} and \ref{alg}.\qed
\end{proof}

\begin{Remark}[Relationship to other notions of K\"ahlerian twistor spinors]\label{relation}
\normalfont{
The term ``K\"ahlerian twistor spinor" has already been used in the literature. A class of spinors with this name has been introduced by K.-D.~Kirchberg \cite{kirch2} and by O.~Hijazi \cite{hij}. We explain here the relationship between Definition \ref{defkts} and these definitions. 

In \cite{kirch2}, K.-D.~Kirchberg defined  a K\"ahlerian twistor spinor of type $r$ (for $1\leq r\leq m$) to be a spinor $\varphi\in\Gamma(\sigm)$ satisfying the equation
\begin{equation}\label{eckirch}
\nabla_X\varphi=-\frac{1}{4r}(X\cdot D\varphi+JX\cdot D^c\varphi).
\end{equation}
He showed that a solution of \eqref{eckirch} must lie in $\Gamma(\s{r-1}\oplus\s{m-r+1})$. By rewriting this equation using the operators $D^+$ and $D^-$:
\begin{equation*}
\begin{cases}
\nabla_{X^{+}}\varphi=-\frac{1}{2r}X^{+}\cdot D^{-}\varphi,\\
\nabla_{X^{-}}\varphi=-\frac{1}{2r}X^{-}\cdot D^{+}\varphi,
\end{cases}
\end{equation*}
it follows that the spinors satisfying \eqref{eckirch} are exactly the anti-holomorphic and holomorphic K\"ahlerian twistor spinors in $\s{r-1}$, respectively in $\s{m-r+1}$, see Definition \ref{defsemi}. K.-D.~Kirchberg further proved, \cite{kirch1}, some vanishing results for these spinors which we shall also obtain in \S~\ref{kecase} (for K\"ahler-Einstein manifolds) and \S ~\ref{eigenric} (for K\"ahler manifolds of constant scalar curvature) as a special case.

In \cite{hij}, O.~Hijazi considered as defining equation for a spinor $\varphi\in\Gamma(\sigm)$ the following slightly more general equation than \eqref{eckirch}:
\begin{equation}\label{echij}
\nabla_X\varphi=aX\cdot D\varphi+bJX\cdot D^c\varphi,
\end{equation}
where $a$ and $b$ are any real numbers. For example, if $a=-\frac{1}{n}$ and $b=0$, then a solution of \eqref{echij} is a Riemannian twistor spinor (as \eqref{ecriemtwist} shows). Furthermore O.~Hijazi proved that on a K\"ahler spin manifold with nonzero scalar curvature there exists a nontrivial solution of \eqref{echij} if and only if $a=b=-\frac{1}{4(r+1)}$, for some integer $r$ with $0\leq r\leq m-2$, thus reducing  equation \eqref{echij} to \eqref{eckirch}. For $r=m-1$ it is proven (\cite[Theorem~4.30]{hij} and \cite[Proposition~11, Theorem~17]{kirch2}) that the solutions of the equation \eqref{echij} are exactly the Riemannian twistor spinors and they are all trivial on a K\"ahler spin manifold of nonzero scalar curvature.

In order to better compare these definitions we notice that, using the operators $D$ and $D^c$, the defining equation \eqref{defec} for K\"ahlerian twistor spinors can be rewritten as follows:
\[\nabla_X\varphi=-\frac{m+2}{8(r+1)(m-r+1)}(X\cdot D\varphi + JX\cdot D^c\varphi)-\frac{m-2r}{8(r+1)(m-r+1)}i(JX\cdot D\varphi - X\cdot D^c\varphi).\]}
\end{Remark}

An important property of special K\"ahlerian twistor spinors noticed by K.-D.~Kirchberg in \cite{kirch1} is that they are eigenspinors of the square of the Dirac operator, if we assume the scalar curvature to be constant. This is implied by the Lichnerowicz formula as follows.

Let $\varphi\in\akts(r)$: $\nabla_{X}\varphi=-\frac{1}{2(r+1)}X^{-}\cdot D^{+}\varphi$. By differentiating once this defining equation and then contracting, we get
\[\nabla_{e_j}\nabla_{e_j}\varphi=-\frac{1}{2(r+1)}e_j^{-}\cdot\nabla_{e_j}D^{+}\varphi,\]
where $\of$ is a local orthonormal frame parallel at the point where the computations are made. Since $D^-\varphi=0$, it follows
\[\nabla^*\nabla\varphi=\frac{1}{2(r+1)}D^{-}D^{+}\varphi=\frac{1}{2(r+1)}D^2\varphi,\]
which together with the Lichnerowicz formula \eqref{lichn} yields
\begin{equation}\label{semieigen}
D^{2}\varphi=\frac{r+1}{2(2r+1)}S\varphi.
\end{equation}
Thus, if the scalar curvature $S$ is constant, then $\varphi$ is an eigenspinor of $D^2$. Similarly, if $\varphi\in\hkts(r)$, then we get $D^{2}\varphi=\frac{m-r+1}{2(2m-2r+1)}S\varphi$.

\subsection{Particular Cases of K\"ahlerian Twistor Spinors}\label{extrcases}
We first look at extremal cases of K\"ahlerian twistor spinors, \emph{i.e.} of highest and lowest type and notice that they are always special K\"ahlerian twistor spinors. Let $\varphi\in\Gamma(\s{r})$: if $r=0$, then $D^-\varphi$ vanishes automatically and if $r=m$, then $D^+\varphi=0$. Thus $\varphi$ is an anti-holomorphic, respectively holomorphic K\"ahlerian twistor spinor. Moreover, as  shown by K.-D.~Kirchberg, they are exactly the anti-holomorphic, respectively holomorphic sections in $\s{0}$, respectively $\s{m}$ (\cite[Theorem~12]{kirch2}, with the remark that we use different conventions, namely $S_0$ in \cite{kirch2} corresponds to $\s{m}$ in our notation):

\[ \kts(0)=\akts(0)=\bar{H}^0(M, \s{0})=\mathfrak{j}H^0(M, K^{\frac{1}{2}}),\]
\[ \kts(m)=\hkts(m)=H^0(M, \s{m})= H^0(M, K^{\frac{1}{2}}).\]

Another special case is the middle dimension, when $m$ is even and $r=\frac{m}{2}$. This is the only case when the coefficients of the defining equations \eqref{defec} of a K\"ahlerian twistor spinor are equal:

\begin{equation}\label{defecmiddle}
\nabla_{X}\varphi=-\frac{1}{m+2}(X^{+}\cdot D^{-}\varphi+ X^{-}\cdot D^{+}\varphi).
\end{equation}

We show that on a compact K\"ahler spin manifold of positive constant scalar curvature there does not exist any nontrivial solution of this twistorial equation, which means that there are no K\"ahlerian twistor spinors in the middle dimension.

By differentiating and then contracting \eqref{defecmiddle} we obtain
\[\nabla_{e_j}\nabla_{e_j}\varphi=-\frac{1}{m+2}(e_j^{+}\cdot\nabla_{e_j^+} D^{-}\varphi+ e_j^{-}\cdot \nabla_{e_j^-}D^{+}\varphi),\]
\noindent where $\of$ is an orthonormal frame parallel at the point where the computation is made. Thus, it follows that
\[\nabla^*\nabla\varphi=\frac{1}{m+2}(D^+D^-\varphi +D^-D^+\varphi)=\frac{1}{m+2}D^2\varphi,\]
\noindent which together with Lichnerowicz formula \eqref{lichn} implies
\[D^2\varphi=\frac{m+2}{4(m+1)}S\varphi,\]
\noindent showing that if the scalar curvature is constant, then $\varphi$ is an eigenspinor of the eigenvalue $\lambda_0^{even}=\frac{m+2}{4(m+1)}S$, which is strictly smaller than $\frac{m}{4(m-1)}S$. This value is the lower bound given by Kirchberg's inequality \eqref{kircheven} for $m$ even. Thus $\varphi$ must be zero.

If $S=0$, then from the above relations we have $\nabla^*\nabla\varphi=D^2\varphi=0$, so that $\varphi$ is a parallel spinor if the manifold $M$ is compact.

\section{The K\"ahlerian Twistor Connection}\label{sectktc}

The purpose of this section is to construct for each $r$ (from now on we fix an $r$ with $0<r<m$ and $r\neq \frac{m}{2}$) a vector bundle with a connection, called {\it K\"ahlerian twistor connection}, such that K\"ahlerian twistor spinors in the subbundle $\s{r}$ are in one-to-one correspondence to parallel sections of this connection. This allows us to conclude for instance, that the space of K\"ahlerian twistor spinors is finite dimensional. The curvature of this connection provides useful formulas for computations with K\"ahlerian twistor spinors,  needed in \S ~\ref{classif} to describe geometrically the K\"ahler manifolds admitting such spinors.

The idea of constructing a larger vector bundle with a suitable connection such that solutions of a certain equation correspond to parallel sections has often appeared in the literature. For example for Riemannian twistor spinors this construction was done by Th.~Friedrich \cite{fr2} and for conformal Killing forms by U. Semmelmann \cite{uwehabil}.

By definition of K\"ahlerian twistor spinors, the covariant derivative of $\varphi$ involves $\varphi^+:=D^+\varphi$ and $\varphi^-:=D^-\varphi$. Hence, the first step will be the computation of the covariant derivatives of these sections, which yields an expression involving only zero order terms and $D^2\varphi$. Then we compute the covariant derivative of $D^2\varphi$ and get an expression involving zero order terms and the sections $\varphi^+$ and $\varphi^-$, showing that the system closes and thus defines a connection. More precisely, if we denote by $\hat{\varphi}=(\varphi, \varphi^+, \varphi^-,D^2\varphi)\in\Gamma(\s{r}\oplus\s{r+1}\oplus\s{r-1}\oplus\s{r})$, then we have $\nabla_X\hat{\varphi}=B(X)\hat{\varphi}$, where $B(X)$ is a certain $4\times 4$-matrix whose coefficients are endomorphisms of the spinor bundle, depending on the vector field $X$. The K\"ahlerian twistor connection is then a connection on $\s{r}\oplus\s{r+1}\oplus\s{r-1}\oplus\s{r}$, defined as $\hat{\nabla}_X=\nabla_X-B(X)$ and the K\"ahlerian twistor spinors are the first component of parallel sections of $\hat{\nabla}$.

Before proceeding with the computation we make a short digression to introduce the notation and to give some useful formulas needed in the sequel. We consider the local formula \eqref{diracloc} defining the Dirac operator (or the formulas \eqref{defd+d-} defining $D^-$ and $D^+$) and apply it to sections of different associated bundles (and then $\nabla$ is the corresponding connection induced by the Levi-Civita connection). When applied to functions we get: $Df=df,\quad D^{-}(f)=\partial f, \quad  D^{+}(f)=\bar\partial f$.
If we apply \eqref{diracloc} to a vector field $X$ we get the following endomorphisms of the spinor bundle:
\[D(X)\cdot:=\sum_{j=1}^{n}e_j\cdot \nabla_{e_j}X \cdot,\quad D^2(X)\cdot:=\sum_{i,j=1}^{n}e_j\cdot e_i\cdot\nabla_{e_j}\nabla_{e_i}X \cdot.\]
On forms we have: $D=d+\delta,\quad D^+=\bar\partial+\partial^*, \quad D^-=\partial+\bar{\partial}^*$.
We may also extend the formula for the Dirac operator \eqref{diracloc} on endomorphisms of the tangent bundle. For instance, for the Ricci tensor we define the following endomorphism of the spinor bundle:
\[D(\ric)(X)\cdot:=e_i\cdot(\nabla_{e_i}\ric)(X)\cdot.\]

By straightforward computation we obtain:
\begin{Lemma}
The following commutator rules hold for any vector field $X$ on $M$, where $\of$ is a local orthonormal frame parallel at the point where the computations are done:
\begin{equation} \label{dx}
DX\cdot+X\cdot D=D(X)\cdot-2\nabla_X,
\end{equation}
\begin{equation}\label{d+x}
D^+X^+\cdot +X^+\cdot D^+=D^+(X^+)\cdot, \quad \quad D^+X^-\cdot + X^-\cdot D^+=D^+(X^-)\cdot-2\nabla_{X^-},
\end{equation}
\begin{equation}\label{d-x}
D^-X^-\cdot +X^-\cdot D^-=D^-(X^-)\cdot, \quad \quad D^-X^+\cdot + X^+\cdot D^-=D^-(X^+)\cdot-2\nabla_{X^+},
\end{equation}
\begin{equation}\label{nablax}
[\nabla_X,D]=-\frac{1}{2}\ric(X)\cdot-e_i\cdot\nabla_{\nabla_{e_i}X},
\end{equation}
\begin{equation}\label{nablax+}
[\nabla_{X^+},D^+]=-\frac{1}{2}\ric(X^+)\cdot-e_i^+\cdot\nabla_{\nabla_{e_i^-}X^+}, \quad \quad [\nabla_{X^-},D^+]=-e_i^+\cdot\nabla_{\nabla_{e_i^-}X^-},
\end{equation}
\begin{equation}\label{nablax-}
[\nabla_{X^-},D^-]=-\frac{1}{2}\ric(X^-)\cdot-e_i^-\cdot\nabla_{\nabla_{e_i^+}X^-}, \quad \quad [\nabla_{X^+},D^-]=-e_i^-\cdot\nabla_{\nabla_{e_i^+}X^+},
\end{equation}
\begin{equation}\label{d2}
[D^2,X]=-\ric(X)\cdot+D^2(X)\cdot-2\nabla_{e_i}X\cdot\nabla_{e_i},
\end{equation}
\begin{equation}\label{dric}
D(\ric(X)\cdot)+\ric(X)\cdot D=D(\ric(X))\cdot -2\nabla_{\ric(X)},
\end{equation}
\begin{equation}\label{nablad2}
\begin{split}
[\nabla_X,D^2]=&-\frac{1}{2}D(\ric(X))\cdot+\nabla_{\ric(X)}-e_j\cdot e_i\cdot R_{e_j,\nabla_{e_i}X}-e_j\cdot e_i\cdot\nabla_{\nabla_{e_i}X}\nabla_{e_j}\\
&-e_j\cdot e_i\cdot\nabla_{\nabla_{e_j}\nabla_{e_i}X}-e_i\cdot\nabla_{\nabla_{e_i}X}D.\\
\end{split}
\end{equation}
\end{Lemma}

Let now $\varphi$ be a K\"ahlerian twistor spinor in $\s{r}$. First we derive some formulas relating the second order differential operators $D^+D^-$, $D^-D^+$ and $D^2$, when applied to a K\"ahlerian twistor spinor. Since we assumed that $r\neq m/2$, the system formed by the Weitzenb\"ock formula \eqref{wbform} (using the fact that by definition $T_r\varphi=0$) and the Lichnerowicz formula \eqref{lichn} can be inversed and we get the following relations:

\begin{equation}\label{d+d-d2}
D^+D^-\varphi=-\frac{(2r+1)(m-r+1)}{m-2r}D^2\varphi+\frac{(r+1)(m-r+1)}{2(m-2r)}S\varphi,
\end{equation}

\begin{equation}\label{d-d+d2}
D^-D^+\varphi=\frac{(2m-2r+1)(r+1)}{m-2r}D^2\varphi-\frac{(r+1)(m-r+1)}{2(m-2r)}S\varphi,
\end{equation}

\begin{equation}\label{d+d-}
D^+D^-\varphi=-\frac{(2r+1)(m-r+1)}{(2m-2r+1)(r+1)}D^-D^+\varphi +\frac{m-r+1}{2(2m-2r+1)}S\varphi,
\end{equation}

\begin{equation}\label{d-d+}
D^-D^+\varphi=-\frac{(2m-2r+1)(r+1)}{(2r+1)(m-r+1)}D^+D^-\varphi +\frac{r+1}{2(2r+1)}S\varphi.
\end{equation}

We now compute the covariant derivatives of $D^+\varphi$ and $D^-\varphi$ in the direction of a vector field $X$ which is parallel at the point where the computations are done. The local orthonormal frame $\of$ is parallel at this point too.
\begin{equation*}
\begin{split}
\nabla_{X^+}(D^+\varphi)\overset{\eqref{nablax+}}{=}&D^+(\nabla_{X^+}\varphi)-\frac{1}{2}\ric(X^+)\cdot\varphi\\
=&-\frac{1}{2(m-r+1)}D^+(X^+\cdot D^-\varphi)-\frac{1}{2}\ric(X^+)\cdot\varphi\\
\overset{\eqref{d+x}}{=}&\frac{1}{2(m-r+1)}X^+\cdot D^+D^-\varphi-\frac{1}{2}\ric(X^+)\cdot\varphi\\
\overset{\eqref{d+d-d2}}{=}&-\frac{2r+1}{2(m-2r)}X^+\cdot D^2\varphi+\frac{r+1}{4(m-2r)}SX^+\cdot\varphi-\frac{1}{2}\ric(X^+)\cdot\varphi,
\end{split}
\end{equation*}

\begin{equation*}
\begin{split}
\nabla_{X^-}(D^+\varphi)\overset{\eqref{nablax-}}{=}&D^+(\nabla_{X^-}\varphi)=-\frac{1}{2(r+1)}D^+(X^-\cdot D^+\varphi)\\
\overset{\eqref{d+x}}{=}&\frac{1}{2(r+1)}[X^-\cdot D^+(D^+\varphi)+2\nabla_{X^-}(D^+\varphi)]=\frac{1}{r+1}\nabla_{X^-}(D^+\varphi),
\end{split}
\end{equation*}
so that $\nabla_{X^-}(D^+\varphi)=0$.

\noindent These two equations give the second row of the connection in \eqref{matrix1}. Similarly, for the covariant derivative of $D^-\varphi$ we obtain
 \begin{equation*}
 \begin{split}
\nabla_{X^-}(D^-\varphi)&=\frac{2m-2r+1}{2(m-2r)}X^-\cdot D^2\varphi-\frac{m-r+1}{4(m-2r)}SX^-\cdot\varphi-\frac{1}{2}\ric(X^-)\cdot\varphi,\\
\nabla_{X^+}(D^-\varphi)&=0,
\end{split}
\end{equation*}

\noindent which yield the third row of the connection in \eqref{matrix1}.

\noindent For the last component of the connection we compute the covariant derivative of $D^2\varphi$.
\begin{equation}\label{lastrow}
\begin{split}
\nabla_X(D^2\varphi)&\overset{\eqref{nablad2}}{=}D^2(\nabla_X\varphi)-\frac{1}{2}D(\ric)(X)\cdot\varphi+\nabla_{\ric(X)}\varphi-e_j\cdot e_i\cdot\nabla_{\nabla_{e_j}\nabla_{e_i}X}\varphi.
\end{split}
\end{equation}

We now compute separately the terms appearing in \eqref{lastrow}. The first one is given by

\begin{equation}\label{first}
\begin{split}
D^2(\nabla_X\varphi)&=-\frac{1}{2(m-r+1)}D^2(X^+\cdot D^-\varphi)-\frac{1}{2(r+1)}D^2(X^-\cdot D^+\varphi),
\end{split}
\end{equation}

\noindent where 

\begin{equation*}
\begin{split}
D^2(X^+\cdot D^-\varphi)\overset{\eqref{d2}}{=}&X^+\cdot D^2D^-\varphi-\ric(X^+)\cdot D^-\varphi+D^2(X^+)\cdot D^-\varphi\\
\overset{\eqref{d2sum}}{=}&X^+\cdot D^-D^+D^-\varphi-\ric(X^+)\cdot D^-\varphi+D^2(X^+)\cdot D^-\varphi\\
\overset{\eqref{d+d-}}{=}&\frac{m-r+1}{2(2m-2r+1)}X^+\cdot D^-(S\varphi)-\ric(X^+)\cdot D^-\varphi+D^2(X^+)\cdot D^-\varphi\\
=&\frac{m-r+1}{2(2m-2r+1)}S X^+\cdot D^-\varphi-\frac{m-r+1}{2(2m-2r+1)} \partial(S)\cdot X^+\cdot \varphi\\
&-\frac{m-r+1}{2m-2r+1}X^+(S)\varphi-\ric(X^+)\cdot D^-\varphi+D^2(X^+)\cdot D^-\varphi
\end{split}
\end{equation*}

\noindent and similarly
\begin{equation*}
\begin{split}
D^2(X^-\cdot D^+\varphi)=&\frac{r+1}{2(2r+1)}S X^-\cdot D^+\varphi-\frac{r+1}{2(2r+1)} \bar\partial(S)\cdot X^-\cdot \varphi\\
&-\frac{r+1}{2r+1}X^-(S)\varphi-\ric(X^-)\cdot D^+\varphi+D^2(X^-)\cdot D^+\varphi.
\end{split}
\end{equation*}

\noindent The third term in \eqref{lastrow} is given by
\begin{equation}\label{rric}
\begin{split}
\nabla_{\ric(X)}\varphi&=\nabla_{\ric(X^+)}\varphi+\nabla_{\ric(X^-)}\varphi\\
&=-\frac{1}{2(m-r+1)}\ric(X^{+})\cdot D^{-}\varphi-\frac{1}{2(r+1)}\ric(X^{-})\cdot D^{+}\varphi.
\end{split}
\end{equation}

\noindent and the last term in \eqref{lastrow} is
\begin{equation}\label{mixt}
e_j\cdot e_i\cdot\nabla_{\nabla_{e_j}\nabla_{e_i}X}\varphi=-\frac{1}{2(m-r+1)}D^2(X^+)\cdot D^-\varphi-\frac{1}{2(r+1)}D^2(X^-)\cdot D^+\varphi.
\end{equation}

Substituting \eqref{first}, \eqref{rric} and \eqref{mixt} in \eqref{lastrow}, we get the following equality, which yields the last row of the connection matrix \eqref{matrix1}:
\begin{equation}
\begin{split}
\nabla_X(D^2\varphi)=&-\frac{1}{2}D(\ric)(X)\cdot\varphi+\frac{1}{2(2m-2r+1)}X^+(S)\varphi+\frac{1}{2(2r+1)}X^-(S)\varphi\\
&+\frac{1}{4(2r+1)} \bar\partial(S)\cdot X^-\cdot \varphi+\frac{1}{4(2m-2r+1)} \partial(S)\cdot X^+\cdot\varphi\\
&-\frac{1}{4(2r+1)}S X^-\cdot D^+\varphi-\frac{1}{4(2m-2r+1)}S X^+\cdot D^-\varphi.
\end{split}
\end{equation}

Hence, we have shown that if $\varphi\in\Gamma(\s{r})$ is a K\"ahlerian twistor spinor, then the four-tuple $(\varphi, \varphi^+:=D^+\varphi, \varphi^-:=D^-\varphi, D^2\varphi)$ is parallel with respect to the following connection, which we call {\it K\"ahlerian twistor connection}:
\begin{equation}\label{matrix1}
\hat{\nabla}_X=
\renewcommand{\arraystretch}{1.8}
\begin{pmatrix}
\nabla_X & \frac{1}{2(r+1)} X^{-}\cdot & \frac{1}{2(m-r+1)} X^{+}\cdot & 0 \\
-\frac{r+1}{4(m-2r)} S X^{+}\cdot+\frac{1}{2} \ric(X^{+})\cdot & \nabla_X & 0 & \frac{2r+1}{2(m-2r)}X^{+}\cdot\\
\frac{m-r+1}{4(m-2r)} S X^{-}\cdot+\frac{1}{2} \ric(X^{-})\cdot & 0 & \nabla_X & -\frac{2m-2r+1}{2(m-2r)}X^{-}\cdot\\
A(X) & \frac{1}{4(2r+1)}S X^{-}\cdot& \frac{1}{4(2m-2r+1)}S X^{+}\cdot& \nabla_X
\end{pmatrix},
\end{equation}

\noindent where we denote by $A(X)$ the following endomorphism of the spinor bundle
\begin{equation*}
\begin{split}
A(X):=&\frac{1}{2}D(\ric)(X)-\frac{1}{4(2m-2r+1)}(dS)^-\cdot X^{+}-\frac{1}{4(2r+1)}(dS)^+\cdot X^{-}\\
&-\frac{1}{2(2m-2r+1)}X^{+}(S)-\frac{1}{2(2r+1)}X^{-}(S).
\end{split}
\end{equation*}

Moreover, it follows that any parallel section is of the form $(\varphi, \varphi^+, \varphi^-, D^2\varphi)$ with $\varphi$ a K\"ahlerian twistor spinor:

\begin{Proposition}\label{1-1}
There is a one-to-one correspondence between K\"ahlerian twistor spinors in $\s{r}$ and parallel sections of the bundle $\s{r}\oplus\s{r+1}\oplus\s{r-1}\oplus\s{r}$ with respect to the connection $\hat{\nabla}$ given by \eqref{matrix1}. The explicit bijection is given by
\[\varphi \mapsto  \hat{\varphi}=(\varphi, \varphi^+, \varphi^-,D^2\varphi).\]
\end{Proposition}
\begin{proof}
In the above discussion we have seen that the function $\varphi \mapsto  \hat{\varphi}$ takes values in the space of parallel sections with respect to the connection \eqref{matrix1} and it is obviously injective. Thus we only need to prove its surjectivity.

Let $(\varphi,\psi,\xi,\eta)\in\Gamma(\s{r}\oplus\s{r+1}\oplus\s{r-1}\oplus\s{r})$ be a parallel section with respect to \eqref{matrix1}. Since the first row of this connection is exactly the K\"ahlerian twistor operator, it follows by contractions that the first three components are $(\varphi, \varphi^+, \varphi^-)$, where $\varphi$ is a K\"ahlerian twistor spinor. From
\[\nabla_{X^-}\varphi=-\frac{1}{2(r+1)} X^{-}\cdot\psi\]
we get by contraction that $\psi=D^-\varphi$. Similarly, from
\[\nabla_{X^+}\varphi=-\frac{1}{2(m-r+1)} X^{+}\cdot\xi\]
we get $\xi=D^+\varphi$. Substituting now $\psi$ and $\xi$ in the first row yields that $\varphi$ is a K\"ahlerian twistor spinor. For the last component of the four-tuple we may compare for example the second row of the connection matrix applied to the parallel sections $(\varphi,\psi= \varphi^+,\xi= \varphi^-,\eta)$ and $(\varphi, \varphi^+, \varphi^-,D^2\varphi)$ obtaining:
\[X^+\cdot\eta=X^+\cdot D^2\varphi,\]
which contracted yields $\eta=D^2\varphi$.\qed
\end{proof}

If the manifold $(M,g,J)$ is K\"ahler-Einstein, then $\ric(X)=\frac{S}{n}X$ and the K\"ahlerian twistor connection $\hat{\nabla}$ simplifies as follows:
\begin{equation}
\renewcommand{\arraystretch}{1.8}
\hat{\nabla}_X=
\begin{pmatrix}
\nabla_X & \frac{1}{2(r+1)}X^{-}\cdot & \frac{1}{2(m-r+1)} X^{+}\cdot & 0 \\
-\frac{r(m+2)}{4m(m-2r)}S X^{+}\cdot & \nabla_X & 0 & \frac{2r+1}{2(m-2r)}X^{+}\cdot \\
\frac{(m-r)(m+2)}{4m(m-2r)}S X^{-}\cdot & 0 & \nabla_X & -\frac{2m-2r+1}{2(m-2r)} X^{-}\cdot \\
0 & \frac{1}{4(2r+1)}S X^{-}\cdot & \frac{1}{4(2m-2r+1)}S X^{+}\cdot & \nabla_X
\end{pmatrix}.
\end{equation}

\section{The Curvature of the K\"ahlerian Twistor Connection}\label{sectcurv}

In this section we compute the curvature of the K\"ahlerian twistor connection. The first component of this curvature allows us to reduce the K\"ahlerian twistor connection $\hat{\nabla}$ to one acting on a bundle of smaller rank, namely on $\s{r}\oplus\s{r+1}\oplus\s{r-1}$, which is given by the matrix \eqref{matrix}. 

Let $\varphi\in\Gamma(\s{r})$ be a K\"ahlerian twistor spinor. Since $\hat{\varphi}=(\varphi, \varphi^+, \varphi^-, \eta:=D^2\varphi)$ is a parallel section of $\hat{\nabla}$ (Proposition~\ref{1-1}), then by definition the curvature of this connection vanishes on this section: $\hat{R}_{X,Y}(\hat\varphi)=0$, for any vector fields $X$ and $Y$. Thus, computing the components of $\hat{R}$ we get certain identities which by further contractions yield the formulas in Proposition~\ref{formtwists}. 

By straightforward computation it follows that the first component of $\hat{R}$ is given by:
\begin{equation*}
\begin{split}
\hat{R}_{X,Y}\begin{pmatrix} \varphi\\ \varphi^+\\ \varphi^-\\ \eta\end{pmatrix}_1
=& R_{X,Y}\varphi - \frac{S}{4(m-2r)} i<X,JY>\varphi+\frac{1}{4(r+1)}[X^-\cdot \ric(Y^+)-Y^-\cdot \ric(X^+)]\cdot\varphi\\
&+\frac{1}{4(m-r+1)}[X^+\cdot \ric(Y^-)-Y^+\cdot \ric(X^-)]\cdot\varphi\\
&+\frac{1}{4(m-2r)}[\frac{2r-m}{(r+1)(m-r+1)}(X^+\cdot Y^- - Y^+\cdot X^-)+\frac{2(2r+1)}{r+1}i<X,JY>]\cdot\eta.
\end{split}
\end{equation*}

We now need the following formulas for contractions, which hold for endomorphisms of the spinor bundle restricted to $\mathrm{\Sigma}_r\mathrm{M}$:
\begin{equation}\label{contr1}
e_i^+\cdot e_i^-=-2r, \quad e_i^-\cdot e_i^+=-2(m-r),
\end{equation}
\begin{equation}\label{contr2}
e_i\cdot R_{e_i,Y}=\frac{1}{2}\ric(Y), \quad e_i^-\cdot R_{e_i^+,Y}=\frac{1}{4}[\ric(Y)+i\rho(Y)],
\end{equation}
\begin{equation}\label{contr3}
e_i\cdot \ric(e_i)=-S, \quad e_i^-\cdot \ric(e_i^+)=-\frac{S}{2}-i\rho,
\end{equation}
where $\rho$ is the Ricci form: $\rho(X,Y)=\ric(JX,Y)$. Using \eqref{contr1}-\eqref{contr3} we obtain the following formula for the contraction of the first component of the curvature: 
\begin{equation*}
\begin{split}
0=&[-\frac{S}{4(m-2r)}(Y^- -Y^+)-\frac{1}{4(r+1)}(\frac{S}{2}Y^- +i Y^-\cdot\rho)-\frac{1}{4(m-r+1)}(\frac{S}{2}Y^+ -i Y^+\cdot\rho)]\cdot\varphi\\
&+\frac{1}{4(m-2r)}[\frac{2r-m}{(r+1)(m-r+1)}(-2(m-r+1)Y^- -2rY^+)+\frac{2(2r+1)}{r+1}(Y^- -Y^+)]\cdot\eta\\ 
=& -\frac{(m+2)S}{8(r+1)(m-2r)}Y^-\cdot\varphi +\frac{(m+2)S}{8(m-r+1)(m-2r)}Y^+\cdot\varphi-\frac{i}{4(r+1)}Y^-\cdot\rho\cdot\varphi\\
&+\frac{i}{4(m-r+1)}Y^+\cdot\rho\cdot\varphi+\frac{m+1}{2(r+1)(m-2r)}Y^-\cdot\eta-\frac{m+1}{2(m-r+1)(m-2r)}Y^+\cdot\eta,
\end{split}
\end{equation*}
or equivalently:
\begin{equation*}
\frac{(m+2)S}{4(m-2r)}Y^-\cdot\varphi +\frac{i}{2}Y^-\cdot\rho\cdot\varphi-\frac{m+1}{m-2r}Y^-\cdot\eta=0,
\end{equation*}
\begin{equation*}
\frac{(m+2)S}{4(m-2r)}Y^+\cdot\varphi +\frac{i}{2}Y^+\cdot\rho\cdot\varphi-\frac{m+1}{m-2r}Y^+\cdot\eta=0,
\end{equation*}
which both yield by a further contraction:
\begin{equation*}
\frac{(m+2)S}{4(m-2r)}\cdot\varphi +\frac{i}{2}\cdot\rho\cdot\varphi-\frac{m+1}{m-2r}\cdot\eta=0,
\end{equation*}
so that 
\begin{equation}\label{eta}
D^2\varphi=\eta=\frac{(m+2)S}{4(m+1)}\varphi +\frac{m-2r}{2(m+1)}i\rho\cdot\varphi.
\end{equation}

This relation allows us to reduce the connection to one acting on sections of the vector bundle $\s{r}\oplus \s{r+1}\oplus \s{r-1}$. Thus, substituting \eqref{eta} in the second and third row of the connection \eqref{matrix1}, we get the following connection with respect to which the triple $(\varphi, \varphi^+, \varphi^-)$ is parallel for any K\"ahlerian twistor spinor $\varphi\in\Gamma(\s{r})$:

\begin{equation}\label{matrix}
\renewcommand{\arraystretch}{1.8}
\tilde{\nabla}_X=
\begin{pmatrix}
\nabla_X & \frac{1}{2(r+1)} X^{-}\cdot & \frac{1}{2(m-r+1)} X^{+}\cdot \\
-\frac{1}{8(m+1)}S X^+\cdot +\frac{1}{2} \ric(X^{+})\cdot+\frac{2r+1}{4(m+1)}iX^{+}\cdot\rho\cdot & \nabla_X & 0 \\
-\frac{1}{8(m+1)}S X^-\cdot +\frac{1}{2} \ric(X^{-})\cdot -\frac{2m-2r+1}{4(m+1)}iX^{-}\cdot \rho\cdot & 0 & \nabla_X 
\end{pmatrix}.
\end{equation}

As in Proposition~\ref{1-1}, it follows that there is a one-to-one correspondence between K\"ahlerian twistor spinors on $\s{r}$ and parallel sections of the bundle $\s{r}\oplus\s{r+1}\oplus\s{r-1}$ with respect to the connection $\tilde{\nabla}$ given by \eqref{matrix}. An immediate consequence is that the space of K\"ahlerian twistor spinors is finite dimensional and an upper bound for its dimension is given as follows.

\begin{Corollary}\label{cordim}
Let $(M,g,J)$ be a connected spin K\"ahler manifold. The dimension of the space of K\"ahlerian twistor spinors in $\s{r}$ is bounded by the rank of the vector bundle $\s{r}\oplus\s{r+1}\oplus\s{r-1}$:
\[\dim_{\mathbb{C}}(\kts(r))\leq \binom{m}{r}+\binom{m}{r+1}+\binom{m}{r-1}.\]
\end{Corollary}

\begin{Remark}
\normalfont{
Twistor operators are one of the typical examples of Stein-Weiss operators. It was shown by Th.~Branson, \cite{branson}, that they are elliptic. Hence, it follows directly that on compact spin K\"ahler manifolds the space of K\"ahlerian twistor spinors is finite dimensional. However,  our Corollary~\ref{cordim} is a purely local result: the manifold $M$ is not assumed to be compact.}
\end{Remark}

The second component of the curvature of the connection $\hat{\nabla}$ is given by straightforward computation as follows:
\begin{equation*}
\begin{split}
\hat{R}_{X,Y}\begin{pmatrix} \varphi\\ \varphi^+\\ \varphi^- \end{pmatrix}_2=& R_{XY}\varphi^+ + \frac{1}{4(r+1)}[-\frac{1}{4(m+1)} S(X^+\cdot Y^{-}-Y^+\cdot X^-)\\&+(\ric(X^+)\cdot Y^- -\ric(Y^+)\cdot X^-) +\frac{2r+1}{2(m+1)}i(X^+\cdot\rho\cdot Y^- - Y^+\cdot\rho\cdot X^-)]\cdot\varphi^+\\
&+ \frac{1}{4(m-r+1)}[-\frac{1}{4(m+1)}S (X^+\cdot Y^+ - Y^+\cdot X^+) \\
&+ (\ric(X^{+})\cdot Y^+ - \ric(Y^{+})\cdot X^+) +\frac{2r+1}{2(m+1)}i(X^+\cdot\rho\cdot Y^+ - Y^+\cdot\rho\cdot X^+]\cdot\varphi^-\\
&+ [-\frac{1}{8(m+1)}(X(S) Y^+ -Y(S) X^+) +\frac{1}{2} ((\nabla_X \ric)(Y^{+})-(\nabla_Y \ric)(X^{+}))\\
&+\frac{2r+1}{4(m+1)}i(Y^{+}\cdot\nabla_X\rho-X^{+}\cdot\nabla_Y\rho)]\cdot\varphi.
\end{split}
\end{equation*}

Contracting this equation using the formulas \eqref{contr1}-\eqref{contr3} we obtain
\begin{equation*}
\begin{split}
0=&-\frac{1}{4(r+1)(m+1)}[\frac{S}{2}(rY^+ + (r+1)Y^-)+i(r(2r+1)Y^+ + (r+1)(2m-2r+1)Y^-)\cdot\rho\\
&+2(r+1)(m-2r)\ric(Y^-)]\cdot\varphi^+\\
&+\frac{1}{4(m-r+1)(m+1)}[\frac{mS}{2}Y^+ + (4r^2+4r-4rm-3m)iY^+\cdot\rho\\
&+2(m-r-1)(2r-m)\ric(Y^+)]\cdot\varphi^-\\
&+\frac{1}{4(m+1)}[-(m-r-1)Y^+(S)+(r+1)Y^-(S)-rY^+\cdot (dS)^+ + (r+1)Y^+\cdot(dS)^-\\
&+2(2r+1)(m-r-1)i\nabla_{Y^+}\rho+2(r+1)(2m-2r+1)i\nabla_{Y^-}\rho]\cdot\varphi.
\end{split}
\end{equation*}

By projecting this equality onto $\s{r}$ and $\s{r+2}$, it becomes equivalent to the following two equations:
\begin{equation}\label{eec2}
\begin{split}
&-[\frac{S}{2}Y^- +i(2m-2r+1)Y^-\cdot\rho+2(m-2r)\ric(Y^-)]\cdot\varphi^+\\
&+\frac{1}{m-r+1}[\frac{mS}{2}Y^+ + (4r^2+4r-4rm-3m)iY^+\cdot\rho+2(m-r-1)(2r-m)\ric(Y^+)]\cdot\varphi^-\\
&+[-(m-r-1)Y^+(S)+(r+1)Y^-(S) + (r+1)Y^+\cdot(dS)^- \\
&+2(2r+1)(m-r-1)i\nabla_{Y^+}\rho+2(r+1)(2m-2r+1)i\nabla_{Y^-}\rho]\cdot\varphi=0,
\end{split}
\end{equation}

\begin{equation}\label{eec1}
[\frac{S}{2}Y^+ +i(2r+1)Y^+\cdot\rho]\cdot\varphi^+ +(r+1)Y^+\cdot (dS)^+\cdot\varphi=0.
\end{equation}

The contraction of \eqref{eec1} yields ( if $r\neq m-1$):
\begin{equation*}
[\frac{S}{2} +(2r+1)i\rho]\cdot\varphi^+ +(r+1)(dS)^+\cdot\varphi=0,
\end{equation*}

so that

\begin{equation}\label{eec11}
i\rho\cdot\varphi^+=-\frac{1}{2(2r+1)}S\varphi^+-\frac{r+1}{2r+1}(dS)^+\cdot\varphi.
\end{equation}

The contraction of \eqref{eec2} yields:
\begin{equation*}
\begin{split}
0&=(m-r+1)S\varphi^++2(2r+1)(m-r+1)i\rho\cdot\varphi^+\\
&+\frac{2(r-m)(r+1)S}{m-r+1}\varphi^-+\frac{4(r+1)(r-m)(2r-2m-1)}{m-r+1}i\rho\cdot\varphi^-\\
&+2(r+1)(m-r+1)(dS)^+\cdot\varphi -4(r+1)(m-r)(dS)^-\cdot\varphi,
\end{split}
\end{equation*}
\noindent which, by projections onto $\s{r+1}$, respectively $\s{r-1}$, is equivalent to the following two equations:
\begin{equation}\label{eec21}
i\rho\cdot\varphi^+=-\frac{1}{2(2r+1)}S\varphi^+-\frac{r+1}{2r+1}(dS)^+\cdot\varphi,
\end{equation}
\noindent which is again the equation \eqref{eec11} (here it follows to be true also for $r=m-1$) and
\begin{equation}\label{eec22}
i\rho\cdot\varphi^-=\frac{1}{2[2(m-r)+1]}S\varphi^-+\frac{m-r+1}{2(m-r)+1}(dS)^-\cdot\varphi.
\end{equation}

Inserting the relations \eqref{eec21} and \eqref{eec22} back in \eqref{eec2}, it becomes equivalent to the following two equations
\begin{equation}\label{phi+}
\begin{split}
i\nabla_{Y^+}\rho\cdot\varphi=&-\frac{(m-2r)}{(2r+1)(m-r+1)}[\frac{1}{2(2m-2r+1)}SY^+ -\ric(Y^+)]\cdot\varphi^-\\
&+\frac{1}{2(2r+1)}Y^+(S)\cdot\varphi+\frac{1}{2(2m-2r+1)}Y^+\cdot(dS)^-\cdot\varphi,
\end{split}
\end{equation}

\begin{equation}\label{phi-}
\begin{split}
i\nabla_{Y^-}\rho\cdot\varphi=&-\frac{m-2r}{(r+1)(2m-2r+1)}[\frac{S}{2(2r+1)}Y^- -\ric(Y^-)]\cdot\varphi^+\\
&-\frac{1}{2(2m-2r+1)}Y^-(S)\cdot\varphi-\frac{1}{2(2r+1)}Y^-\cdot(dS)^+\cdot\varphi.
\end{split}
\end{equation}

A similar computation for the third component of the curvature of the K\"ahlerian twistor connection only yields again the equations \eqref{eec21}, \eqref{eec22}, \eqref{phi+} and \eqref{phi-}.

We gather now the formulas that we obtained for the actions on a special K\"ahlerian twistor spinor and deduce some new ones. Is is enough to consider anti-holomorphic K\"ahlerian twistor spinors, since similar formulas are then fulfilled by holomorphic K\"ahlerian twistor spinors. In fact, they are obtained by conjugating the ones for anti-holomorphic K\"ahlerian twistor spinors and replacing the constant $k=\frac{1}{2(2r+1)}$ with $k=\frac{1}{2(2m-2r+1)}$.

\begin{Proposition}\label{formtwists}
Let $(M,g,J)$ be a spin K\"ahler manifold and $\varphi\in\Gamma(\s{r})$ be an anti-holomorphic K\"ahlerian twistor spinor for some fixed $r$ with $0<r<m$: 
\begin{equation*}
\begin{cases}
\nabla_{X^-}\varphi=-\frac{1}{2(r+1)}X^-\cdot\varphi^+,\\
\nabla_{X^+}\varphi=0,
\end{cases}
\end{equation*}
\noindent so that in particular $\varphi^-=D^-\varphi=0$. Then the following formulas hold, where we denote by $k:=\frac{1}{2(2r+1)}$:
\begin{equation}\label{ds-}
(dS)^-\cdot\varphi=0,
\end{equation}
\begin{equation}\label{ld2}
D^2\varphi=k(r+1)S\varphi,
\end{equation}
\begin{equation}\label{lnab}
\nabla_X\varphi^+=-\frac{1}{2}\ric(X^+)\cdot\varphi,
\end{equation}
\begin{equation}\label{lric}
\ric(X^-)\cdot\varphi=kS X^-\cdot\varphi,
\end{equation}
\begin{equation}\label{lrho}
i\rho\cdot\varphi=kS\varphi,
\end{equation}
\begin{equation}\label{lnab-rho}
i\nabla_{X^+}\rho\cdot\varphi=kX^+(S)\varphi,
\end{equation}
\begin{equation}\label{lrho-}
i\rho\cdot\varphi^+=-kS\varphi^+ -2k(r+1)(dS)^+\cdot\varphi,
\end{equation}
\begin{equation}\label{lric-}
\ric(X^-)\cdot \varphi^+=kS X^-\cdot\varphi^+ - 2k(r+1)X^-(S)\varphi,
\end{equation}
\begin{equation}\label{lnabrho-}
i\nabla_{X^+}\rho\cdot\varphi^+=-\ric^2(X^+)\cdot\varphi+kS\ric(X^+)\cdot\varphi-2k(r+1)\nabla_{X^+}(dS)^+\cdot\varphi-kX^+(S)\varphi^+,
\end{equation}
\begin{equation}\label{lnabrho+}
i\nabla_{X^-}\rho\cdot\varphi^+=-2k(r+1)\nabla_{X^-}(dS)^+\cdot\varphi -3kX^-(S)\varphi^+  -kX^-\cdot(dS)^+\cdot\varphi^+.
\end{equation}
\end{Proposition}
\begin{proof} Equation \eqref{ds-} follows directly from \eqref{eec22} and $\varphi^-=0$. Equation \eqref{ld2} is the property of a special K\"ahlerian twistor spinor to be an eigenspinor of $D^2$, which we have shown in \eqref{semieigen}.

By substituting \eqref{ld2} into the second row of the connection $\hat\nabla$ given by \eqref{matrix1} we get \eqref{lnab}. Similarly, by substituting \eqref{ld2} into the third row of the connection $\hat\nabla$ given by \eqref{matrix1} we get equation \eqref{lric}. Equations \eqref{eta} and \eqref{ld2} yield \eqref{lrho}.

Substituting \eqref{ds-} and $\varphi^-=0$ in \eqref{phi+}  yields \eqref{lnab-rho}. Differentiating \eqref{lrho} we get
\begin{equation*}
\begin{split}
i\nabla_{X^-}\rho\cdot\varphi=&-i\rho\cdot\nabla_{X^-}\varphi+kX^-(S)\varphi+kS\nabla_{X^-}\varphi\\
=&kX^-(S)\varphi-\frac{1}{2(r+1)}kSX^-\cdot\varphi^+ +\frac{i}{2(r+1)}[X^-\cdot\rho -2i \ric(X^-)]\cdot\varphi^+\\
\overset{\eqref{lrho-}}{=}&kX^-(S)\varphi-\frac{1}{2(r+1)}kSX^-\cdot\varphi^+ -\frac{1}{2(r+1)}[kSX^-\cdot\varphi^+\\
&+2k(r+1)X^-\cdot(dS)^+\cdot\varphi -2 \ric(X^-)\cdot\varphi^+]\\
=&kX^-(S)\varphi-\frac{1}{r+1}[kSX^--\ric(X^-)]\cdot\varphi^+ -kX^-\cdot(dS)^+\cdot\varphi,
\end{split}
\end{equation*}

\noindent which compared with \eqref{phi-} yields 
\begin{equation*}
\begin{split}
-kX^-(S)\varphi+\frac{1}{r+1}[kSX^--\ric(X^-)]\cdot\varphi^+=&\frac{(m-2r)}{(2m-2r+1)(r+1)}[kSX^- -\ric(X^-)]\cdot\varphi^+\\
&+\frac{1}{2(2m-2r+1)}X^-(S)\varphi,
\end{split}
\end{equation*}

\noindent thus proving \eqref{lric-}. Equation \eqref{lrho-} is just \eqref{eec21}. Differentiating \eqref{lrho-} we get
\begin{equation*}
\begin{split}
i\nabla_X\rho\cdot\varphi^+=&-i\rho\cdot\nabla_X\varphi^+ -kX(S)\varphi^+ -kS\nabla_X\varphi^+ -2k(r+1)\nabla_X(dS)^+\cdot\varphi\\
&-2k(r+1)(dS)^+\cdot\nabla_X\varphi\\
\overset{\eqref{lnab}}{=}&\frac{1}{2}i\rho\cdot \ric(X^+)\cdot\varphi -kX(S)\varphi^+ +\frac{1}{2}kS \ric(X^+)\cdot\varphi\\ &-2k(r+1)\nabla_X(dS)^+\cdot\varphi+k(dS)^+\cdot X^-\cdot\varphi^+,
\end{split}
\end{equation*}
which together with the commutator relation
\[i\rho\cdot \ric(X^+)\cdot\varphi=i\ric(X^+)\cdot\rho\cdot\varphi -2\ric^2(X^+)\cdot\varphi\]
and \eqref{lrho} yields \eqref{lnabrho-} and \eqref{lnabrho+}.\qed
\end{proof}

\section{The Geometric Description}\label{classif}

In this section we describe geometrically the simply connected compact spin K\"ahler manifolds of constant scalar curvature admitting K\"ahlerian twistor spinors (Theorem~\ref{class}). Let $(M,g,J)$ be such a manifold. The main steps of the proof are: 

\begin{enumerate}
	\item All K\"ahlerian twistor spinors on $M$ are special K\"ahlerian twistor spinors.
	\item The Ricci tensor has two constant eigenvalues.
	\item The Ricci tensor is parallel.
	\item $M=M_1\times M_2$ with $M_1$ a Ricci-flat K\"ahler manifold and $M_2$ K\"ahler-Einstein admitting K\"ahlerian twistor spinors.
\end{enumerate}

\subsection{Special K\"ahlerian Twistor Spinors}

We show that if the scalar curvature is cons\-tant, then each K\"ahlerian twistor spinor is anti-holomorphic or holomorphic. 

\begin{Proposition}\label{allsemi}
Let $(M,g,J)$ be a compact K\"ahler spin manifold of positive constant scalar curvature and $\varphi\in\Gamma(\s{r})$ ($0<r<m$) a K\"ahlerian twistor spinor. Then $\varphi$ is an anti-holomorphic K\"ahlerian twistor spinor if $r<\frac{m}{2}$ or a holomorphic K\"ahlerian twistor spinor if $r>\frac{m}{2}$.
\end{Proposition}
\begin{proof}
The relation \eqref{wbform} and the Lichnerowicz formula \eqref{lichn} imply the following Weitzenb\"ock formula:
\begin{equation}\label{wb}
\frac{2r+1}{2(r+1)}D^-D^+\varphi+\frac{2m-2r+1}{2(m-r+1)}D^+D^-\varphi=\frac{1}{4}S\varphi.
\end{equation}
If $r<\frac{m}{2}$, then $\frac{2r+1}{2(r+1)}<\frac{2m-2r+1}{2(m-r+1)}$ and from \eqref{wb} it follows by integration (where we denote by $||\cdot||$ the global norm: $||\varphi||^2=\int_M<\varphi,\varphi>vol_M$):
\begin{equation}\label{ineg1}
\frac{2r+1}{2(r+1)}||D\varphi||^2\leq\frac{1}{4}S||\varphi||^2,
\end{equation}
with equality if and only if $D^-\varphi=0$. Further it follows that
\[\frac{2(r+1)}{2r+1}\frac{1}{4}S\leq\lambda_{min}\leq\frac{||D\varphi||^2}{||\varphi||^2}\leq\frac{2(r+1)}{2r+1}\frac{1}{4}S,\]
where the first inequality is given by \eqref{valmin1} and the second inequality is the property of the Rayleigh quotient to have as its minimum the smallest eigenvalue $\lambda_{min}$ of the operator (here $D^2$ acting on the Hilbert space of square integrable sections of $\s{r}$).
It then follows that equality must hold in \eqref{ineg1}, so that $\varphi$ must be an anti-holomorphic K\"ahlerian twistor: $D^-\varphi=0$ and in particular an eigenspinor of $D^2$ associated with the smallest eigenvalue of $D^2$ on $\s{r}$ (Proposition~\ref{eigentyper}).

The same argument shows that for $r>\frac{m}{2}$, any K\"ahlerian twistor spinor in $\s{r}$ must be a holomorphic K\"ahlerian twistor spinor.

If $m$ even and $r=\frac{m}{2}$, then we have seen in \S \ref{extrcases} that the only nontrivial K\"ahlerian twistor spinors are the parallel ones. \qed
\end{proof}

\begin{Remark}\label{pos}
\normalfont{
We notice that the condition for the scalar curvature to be positive in Proposition~\ref{allsemi} is not restrictive. If $S\leq 0$, then taking in \eqref{wb} the scalar product with $\varphi$ and integrating over $M$ yields $D^+\varphi=D^-\varphi=0$. As $\varphi$ is also in the kernel of the twistor operator, it follows that $\nabla\varphi=0$, implying that the manifold is Ricci-flat (and thus $S=0$), unless $\varphi\equiv 0$.} Concluding, if the scalar curvature is constant and there exists a nontrivial and non-parallel K\"ahlerian twistor spinor, then the scalar curvature must be positive.
\end{Remark}

\subsection{The Eigenvalues of the Ricci Tensor}\label{eigenric}

We now show that the Ricci tensor of a spin K\"ahler manifold of constant scalar curvature admitting a special K\"ahlerian twistor spinor has two constant eigenvalues and is parallel. This result was proven by A.~Moroianu in \cite{am_even1} and \cite{am_even} for the special case of a limiting even dimensional spin K\"ahler manifold for Kirchberg's inequality \eqref{kircheven}. We note that the same method can be applied when the existence of a special K\"ahlerian twistor spinor is assumed. As shown in Proposition~\ref{allsemi} this is not restrictive, since all K\"ahlerian twistor spinor are special K\"ahlerian twistor spinors if $S$ is constant.

\begin{Theorem}\label{2eigenval}
The Ricci tensor of a K\"ahler spin manifold of constant scalar curvature admitting a nontrivial non-parallel K\"ahlerian twistor spinor has two constant non-negative eigenvalues. More precisely, if $\varphi$ is an anti-holomorphic K\"ahlerian twistor spinor in $\s{r}$, then the Ricci tensor has the eigenvalues $\frac{S}{2(2r+1)}$ and $0$, with multiplicities $2(2r+1)$ and $2(m-2r-1)$ respectively. If $\varphi\in\Gamma(\s{r})$ is a holomorphic K\"ahlerian twistor spinor, then the Ricci tensor has the eigenvalues $\frac{S}{2(2m-2r+1)}$ and $0$, with multiplicities $2(2m-2r+1)$ and $2(2r-m-1)$ respectively. 
\end{Theorem}

Since the multiplicity of an eigenvalue is a positive integer, we get directly the following

\begin{Corollary}
On a K\"ahler spin manifold of constant scalar curvature, anti-holomor\-phic K\"ahlerian twistor spinors may exist only in $\s{r}$ with $r\leq\frac{m-1}{2}$ and holomorphic K\"ahlerian twistor spinors may exist only in $\s{r}$ with $r\geq\frac{m+1}{2}$.
\end{Corollary}

\begin{Remark}
\normalfont{
In the extremal cases, if $m$ is odd and $r=\frac{m\pm1}{2}$, then the existence of a holomorphic (respectively anti-holomorphic) K\"ahlerian twistor spinor $\varphi\in\Gamma(\s{r})$ implies that the Ricci tensor only has one eigenvalue, thus proving that the manifold must be K\"ahler-Einstein. As we move away from the middle dimension the multiplicity of the eigenvalue $0$ of the Ricci tensor grows and the manifold is actually the product of a K\"ahler-Einstein manifold and a Ricci-flat one, if $M$ is supposed to be simply-connected (cf. Theorem~\ref{class}) below.}
\end{Remark}

The proof of Theorem \ref{2eigenval} follows from Lemmas \ref{form1} and \ref{impl}. It is enough to consider anti-holomorphic K\"ahlerian twistor spinors, since the holomorphic ones are obtained by applying the canonical $\mathbb{C}$-anti-linear quaternionic or real structure $\mathfrak{j}$ to the anti-holomorphic ones.

\begin{Lemma}\label{form1}
If $\varphi\in\Gamma(\s{r})$ is an anti-holomorphic K\"ahlerian twistor spinor, then the following formulas hold (with the notation $K=kS=\frac{S}{2(2r+1)}$):
\begin{equation}\label{nab1}
\nabla_X\varphi^+=-\frac{1}{2}\ric(X^+)\cdot\varphi,
\end{equation}
\begin{equation}\label{ric1}
\ric(X^-)\cdot\varphi=K X^-\cdot\varphi, \quad \quad \ric(X^-)\cdot \varphi^+=K X^-\cdot\varphi^+,
\end{equation}
\begin{equation}\label{rho1}
i\rho\cdot\varphi=K\varphi, \quad\quad i\rho\cdot\varphi^+=-K\varphi^+,
\end{equation}
\begin{equation}\label{nab+rho1}
\nabla_{X}\rho\cdot\varphi=0, \quad\quad i\nabla_{X}\rho\cdot\varphi^+=-(\ric^2(X^+)-K\ric(X^+))\cdot\varphi.
\end{equation}
\end{Lemma}

\begin{proof}
These relations follow directly from the general formulas \eqref{lnab} - \eqref{lnabrho+} for anti-holomorphic K\"ahlerian twistor spinors by using the fact that the scalar curvature $S$ is constant. \qed\\
\end{proof}

Let us consider as in \cite{am_even1} the $2$-forms

\[\rho_s:=\frac{1}{2}\sum_{i=1}^{n}e_i\wedge J\ric^s(e_i)=\frac{1}{2}\sum_{i=1}^{n}e_i\cdot J\ric^s(e_i).\]

and the following statements:\\

$(a_s)\quad$ $tr(\ric^s)=2(2r+1)K^s;$

$(b_s)\quad$ $i\rho_s\cdot\varphi=K^s\varphi;$

$(c_s)\quad$ $i\rho_s\cdot\varphi^+=-K^s\varphi^+;$

$(d_s)\quad$ $\nabla_X\rho_s\cdot\varphi=0;$

$(e_s)\quad$ $i\nabla_X\rho_s\cdot\varphi^+=-(\ric^{s+1}(X^+)-K^s\ric(X^+))\cdot\varphi;$

$(f_s)\quad$ $\delta\rho_s=0.$\\

We show by induction that these statements hold for all $s\in \mathbb{N}$. From Lemma \ref{form1} it follows that they are true for $s=1$ (constant scalar curvature implies $\delta\rho=-\frac{1}{2}JdS=0$). We notice that the following result obtained by A. Moroianu in \cite{am_even1} for $r= \frac {m}{2}+1$ is true for any $0 <r<m$:

\begin{Lemma}\label{impl}
The following implications hold:
\begin{enumerate}
	\item $(a_s) \Rightarrow (b_s), (c_s);$
	\item $(b_s), (c_s) \Rightarrow (d_s), (e_s);$
	\item $(a_s), (f_{s-1}) \Rightarrow (f_s);$
	\item $(d_s), (e_s), (f_s)\Rightarrow (a_{s+1}).$
\end{enumerate}
\end{Lemma}

\begin{proof}
Except for the modifications implied by the new values that $r$ may take, the proof given in \cite{am_even1}  works in the same way. Hence, we do not give here the corresponding computations.\qed
\end{proof}

The formulas $(a_s)$ show that the sum of the $s^{th}$ powers of the eigenvalues of $\ric$ equals $2(2r+1)K^s$ for all $s$ and by Newton's relations this proves Theorem~\ref{2eigenval}.

\subsection{K\"ahler-Einstein Manifolds}\label{kecase}

In this section we show that on a K\"ahler-Einstein manifold there may only exist non-extremal K\"ahlerian twistor spinors if $m$ is odd. They must be in $\s{\frac{m-1}{2}}$ or $\s{\frac{m+1}{2}}$ and are automatically K\"ahlerian Killing spinors. Thus such a manifold is a limiting manifold for the Kirchberg's inequality \eqref{kirchodd}. These manifolds have been geometrically described by A.~Moroianu (cf. Theorem~\ref{classodd}). 

Let $(M,g,J)$ be a K\"ahler-Einstein manifold. Then $S$ is constant and by Proposition~\ref{allsemi} it follows that all K\"ahlerian twistor spinors are special K\"ahlerian twistor spinors: if $0\leq r\leq\frac{m}{2}$, then they must be anti-holomorphic K\"ahlerian twistor spinors and if $\frac{m}{2}\leq r\leq m$, then they must be holomorphic K\"ahlerian twistor spinors. As usually, it is sufficient to consider anti-holomorphic K\"ahlerian twistor spinors  $\varphi\in\Gamma(\s{r})$ for a fixed $r$ with $0<r<\frac{m}{2}$.

As $\rho=\frac{1}{2m}S\Omega$, it follows that $i\rho\cdot\varphi=\frac{´m-2r}{2m}S\varphi$ and from \eqref{eta} we get
\begin{equation}\label{eigend2}
D^2\varphi=\frac{m^2+m-2mr+2r^2}{2m(m+1)}S\varphi.
\end{equation}

On the other hand, by Proposition~\ref{eigentyper}, any anti-holomorphic K\"ahlerian twistor spinor in $\s{r}$ (for $0\leq r<\frac{m}{2}$) is an eigenspinor of $D^2$ with the smallest possible eigenvalue on $\s{r}$:
\begin{equation}\label{eigenmic}
D^2\varphi=\frac{(r+1)S}{2(2r+1)}\varphi.
\end{equation}

Comparing the eigenvalues in \eqref{eigend2} and \eqref{eigenmic}, we get for $r\leq \frac{m-1}{2}$:
\[\frac{(r+1)S}{2(2r+1)}=\frac{(m^2+m-2mr+2r^2)S}{2m(m+1)},\]
which, since $S\neq 0$, is equivalent to $0=-r(m-2r)(m-2r-1)$.

As $r\neq 0$ and $r\neq\frac{m}{2}$, it follows that the only possible value for $r$ is $\frac{m-1}{2}$. Thus, except for parallel spinors (if $S=0$) and extremal spinors (those in $\s{0}$), anti-holomorphic K\"ahlerian twistor spinors on a K\"ahler-Einstein manifold can only exist in $\s{\frac{m-1}{2}}$ and they are by definition exactly the K\"ahlerian Killing spinors. A similar result is true for holomorphic K\"ahlerian twistor spinors and it might be obtained by considering the isomorphism \eqref{strj} given by the quaternionic or real structure $\mathfrak{j}$.

Concluding, we have proven the following

\begin{Proposition}\label{ke}
On a spin K\"ahler-Einstein manifold the only nontrivial K\"ahlerian twistor spinors are the extremal ones in $\s{0}$ and $\s{m}$, the K\"ahlerian Killing spinors in $\s{\frac{m-1}{2}}$ and $\s{\frac{m+1}{2}}$ (if $m$ is odd) and the parallel ones (if $M$ is Ricci-flat). 
\end{Proposition}

Combining Proposition~\ref{ke} with the characterization in Proposition~\ref{charactodd} of the limiting mani\-folds of Kirchberg's inequality and their geometric  description given by A.~Moroianu (Theorem~\ref{classodd}) we obtain

\begin{Theorem}\label{classke}
A spin K\"ahler-Einstein manifold admitting nontrivial and non-extremal K\"ahlerian twistor spinors which are not parallel is either $\mathbb{C}P^{4k+1}$ or, in complex dimension $4k+3$, a twistor space over a quaternionic K\"ahler manifold of positive scalar curvature.
\end{Theorem}

\begin{Example}[The complex projective space]
The dimension of the space of K\"ahlerian Killing spinors on $\mathbb{C}P^{m}$ with $m=2k-1$ is $\binom{2k}{k}$(\emph{cf.} \cite{kirch}).
\end{Example}

\begin{Corollary}(cf. \cite{am_th})\label{cordim1}
Let $(M,g,J)$ be a K\"ahler-Einstein manifold admitting nontrivial non-extremal K\"ahlerian twistor spinors, which is not the complex projective space, then the dimension of their space is $2$. More precisely:
\[\dim_{\mathbb{C}}(\kts(\frac{m-1}{2}))=\dim_{\mathbb{C}}(\kts(\frac{m+1}{2}))=1.\]
\end{Corollary}

\subsection{K\"ahlerian Twistor Spinors on K\"ahler Products}

We now study K\"ahlerian twistor spinors on a product of compact spin K\"ahler manifolds and show that they are defined by parallel spinors on one of the factors and special K\"ahlerian twistor spinors on the other factor. For twistor forms a similar result was obtained by A. Moroianu and U. Semmelmann in \cite{ms1}. They showed that twistor forms on a product of compact Riemannian manifolds are defined by Killing forms on the factors. 

Let $M=M_1\times M_2$ be the product of two compact spin K\"ahler manifolds of dimensions $2m$ and $2n$ respectively. Then $M$ is also a spin K\"ahler manifold and its induced spinor bundle is identified with the tensor product of the spinor bundles of the factors: 
\[\sigm\cong \sigm_1\otimes\sigm_2,\]

\noindent with the Clifford multiplication given by:

\[(X_1+X_2)\cdot(\psi_1\otimes\psi_2)=X_1\cdot\psi_1\otimes\psi_2 + \bar{\psi_1}\otimes X_2\cdot\psi_2,\]

\noindent where $\bar\psi$ is the conjugate of the spinor with respect to the decomposition $\sigm_1=\Sigma^+ M_1\oplus\Sigma^- M_1$ given by \eqref{dec+-}. 

We consider the decompositions of the spinor bundles of $M_1$ and $M_2$ with respect to their K\"ahler forms $\Omega_1$, $\Omega_2$ (Lemma~\ref{decompsp}): $\sigm_1=\oplus_{k=0}^{m}\Sigma_k M_1$, $\sigm_2=\oplus_{l=0}^{n}\Sigma_l M_2$. Then the cor\-responding decomposition of $\sigm$ into eigenbundles of $\Omega=\Omega_1+\Omega_2$ is:

\begin{equation}\label{dec1}
\sigm=\oplus_{r=0}^{m+n}\s{r},
\end{equation}
with
\begin{equation}\label{dec2}
\s{r}\cong\oplus_{k=0}^{r}\Sigma_k M_1\otimes\Sigma_{r-k} M_2,
\end{equation}

\noindent since the K\"ahler form $\Omega$ acts on a section of $\Sigma_k M_1\otimes\Sigma_{r-k} M_2$ as:

$\Omega\cdot(\psi_1\otimes\psi_2)=(\Omega_1+\Omega_2)\cdot(\psi_1\otimes\psi_2)=\Omega_1\cdot\psi_1\otimes\psi_2+ \psi_1\otimes\Omega_2\cdot\psi_2=i(2r-m-n)\psi_1\otimes\psi_2$.

Let us define the differential operators:
\[D_1^+=\sum_{i=1}^{2m}e_i^+\cdot\nabla_{e_i^-}, \quad D_2^+=\sum_{j=1}^{2n}f_j^+\cdot\nabla_{f_j^-},\]

\noindent where $\{e_i\}_{i=1,\ldots,2m}$ and $\{f_j\}_{j=1,\ldots,2n}$ denote local orthonormal basis of the tangent distributions to $M_1$, respectively $M_2$. Their adjoints are
\[D_1^-=\sum_{i=1}^{2m}e_i^-\cdot\nabla_{e_i^+}, \quad D_2^-=\sum_{j=1}^{2n}f_j^-\cdot\nabla_{f_j^+}.\]

The following relations are straightforward:
\[D^+=D_1^+ +D_2^+, \quad D^-=D_1^- +D_2^-,\]
\[(D_1^+)^2=(D_2^+)^2=(D_1^-)^2=(D_2^-)^2=0,\]
\[D_1^+ D_2^+ + D_2^+ D_1^+=D_1^- D_2^- + D_2^- D_1^-=0,\]
\[ D_1^+ D_2^- + D_2^- D_1^+=D_1^- D_2^+ + D_2^+ D_1^-=0.\]

We may suppose without loss of generality that one of the factors $M_1$ or $M_2$ is not Ricci-flat. Otherwise, $M=M_1\times M_2$ is Ricci-flat and by the Lichnerowicz formula K\"ahlerian twistor spinors are parallel.

\begin{Theorem}\label{product}
Let $M=M_1\times M_2$ be the product of two compact spin K\"ahler manifolds of dimensions $2m$, respectively $2n$ and suppose that $M_2$ is not Ricci- flat. Let $\psi\in\Gamma(\s{r})$ be a nontrivial K\"ahlerian twistor spinor. Then $\psi$ has the following form
\begin{equation}
\psi=\xi_0\otimes\varphi_r+\xi_m\otimes\varphi_{r-m},
\end{equation}
where $\xi_0, \xi_m$ are parallel spinors in $\mathrm{\Sigma}_0 \mathrm{M}_1$, $\mathrm{\Sigma}_m \mathrm{M}_1$, $\varphi_r$ is an anti-holomorphic K\"ahlerian twistor spinor in $\mathrm{\Sigma}_{r} \mathrm{M}_2$ (if $r\leq n$, otherwise $\varphi_r\equiv 0$) and $\varphi_{r-m}$ is a holomorphic K\"ahlerian twistor spinor in $\mathrm{\Sigma}_{r-m} \mathrm{M}_2$ (if $m\leq r$, otherwise $\varphi_{r-m}\equiv 0$). In particular, $M_1$ is a Ricci-flat manifold and $M_2$ carries special K\"ahlerian twistor spinors in $\mathrm{\Sigma}_{r} \mathrm{M}_2$ or $\mathrm{\Sigma}_{r-m} \mathrm{M}_2$.
\end{Theorem}

\begin{proof}
Let $\psi$ be a K\"ahlerian twistor spinor in $\s{r}$:
\begin{equation}\label{twist}
\begin{cases}
\nabla_{X^+}\psi=-\frac{1}{2(m+n-r+1)}X^+\cdot D^-\psi,\\
\nabla_{X^-}\psi=-\frac{1}{2(r+1)}X^-\cdot D^+\psi,
\end{cases}
\end{equation}
for any vector field $X$ tangent to $M$. With respect to the decomposition \eqref{dec2}, $\psi$ is written as
\[\psi=\psi_0+\cdots+\psi_r,\]
with $\psi_k\in\Gamma(\mathrm{\Sigma}_k \mathrm{M}_1\otimes\mathrm{\Sigma}_{r-k} \mathrm{M}_2)$, for $k=0,\ldots,r$.

Projecting onto the components given by \eqref{dec2}, the twistorial equation \eqref{twist} is equi\-va\-lent to the following two systems of equations:

For $X\in\Gamma(\mathrm{TM}_1)$:
\begin{equation}\label{twist1}
\begin{cases}
\nabla_{X^+}\psi_k=-\frac{1}{2(m+n-r+1)}X^+\cdot (D_1^-\psi_k+D_2^-\psi_{k-1}),\\
\nabla_{X^-}\psi_k=-\frac{1}{2(r+1)}X^-\cdot (D_1^+\psi_k+D_2^+\psi_{k+1})
\end{cases}
\end{equation}
and for $X\in\Gamma(\mathrm{TM}_2)$:
\begin{equation}\label{twist2}
\begin{cases}
\nabla_{X^+}\psi_k=-\frac{1}{2(m+n-r+1)}X^+\cdot (D_1^-\psi_{k+1}+D_2^-\psi_{k}),\\
\nabla_{X^-}\psi_k=-\frac{1}{2(r+1)}X^-\cdot (D_1^+\psi_{k-1}+D_2^+\psi_{k}).
\end{cases}
\end{equation}

If $\{e_i\}_{i=1,\ldots,2m}$ is an orthonormal basis of the $2m$-dimensional manifold $M_1$, then we have on $\mathrm{\Sigma}_r \mathrm{M}_1$:
\[e_i^+\cdot e_i^-=-2r, \quad e_i^-\cdot e_i^+=-2(m-r).\]

By contracting \eqref{twist1} using the relations above, it follows that
\begin{equation*}
D_1^-\psi_k=e_i^-\cdot\nabla_{e_i^+}\psi_k=\frac{m-k+1}{m+n-r+1}(D_1^-\psi_k+D_2^-\psi_{k-1}),
\end{equation*}
\begin{equation*}
D_1^+\psi_k=e_i^+\cdot\nabla_{e_i^-}\psi_k=\frac{k+1}{r+1}(D_1^+\psi_k+D_2^+\psi_{k+1}),
\end{equation*}

so that we get
\begin{equation}\label{relcoef1}
(r-k)D_1^+\psi_k=(k+1)D_2^+\psi_{k+1}
\end{equation}

and
\begin{equation}\label{relcoef2}
(n+k-r)D_1^-\psi_k=(m-k+1)D_2^-\psi_{k-1}.
\end{equation}

We distinguish three cases for $0\leq r\leq m+n$:

\textbf{I.} Suppose that $r$ is strictly smaller than $m$ and $n$. For $k<r$,  \eqref{relcoef1} and \eqref{relcoef2} imply:
\begin{equation}\label{d1d2}
D_1^-D_1^+\psi_k=\frac{k+1}{r-k}D_1^-D_2^+\psi_{k+1}=-\frac{k+1}{r-k}D_2^+D_1^-\psi_{k+1}=-\frac{(k+1)(m-k)}{(r-k)(n+k-r+1)}D_2^+D_2^-\psi_k,
\end{equation}
which integrated over $M$ yields $D_1^+\psi_k=D_2^-\psi_k=0$, $\forall k<r$. Similarly it follows that $D_1^-\psi_k=D_2^+\psi_k=0$, $\forall k>0$. As $D_1^-\psi_0=D_2^-\psi_r=0$ holds automatically, then \eqref{twist1} and \eqref{twist2} show that $\psi_k$ are parallel spinors on $M$ (and thus are zero, since $M$ is not Ricci-flat) for $1\leq k\leq r-1$. The first component $\psi_0\in\Gamma(\mathrm{\Sigma}_0 \mathrm{M}_1\otimes\mathrm{\Sigma}_r \mathrm{M}_2)$ satisfies the equations:
\[\nabla_X\psi_0=0, \quad \text{for all } X\in \Gamma(\mathrm{TM}_1),\]
\[\nabla_{X^+}\psi_0=0, \quad \nabla_{X^-}\psi_0=-\frac{1}{r+1}X^-\cdot D_2^+\psi_0, \quad \text{for all } X\in \Gamma(\mathrm{TM}_2)\]
and $\psi_r\in\Gamma(\mathrm{\Sigma}_r \mathrm{M}_1\otimes\mathrm{\Sigma}_0 \mathrm{M}_2)$ satisfies the equations:

\[\nabla_{X^+}\psi_r=0, \quad \nabla_{X^-}\psi_r=-\frac{1}{r+1}X^-\cdot D_2^+\psi_r, \quad \text{for all } X\in \Gamma(\mathrm{TM}_1),\]
\[\nabla_X\psi_r=0, \quad \text{for all } X\in \Gamma(\mathrm{TM}_2)\]

Thus $\psi_0=\xi_0\otimes\varphi_r$ with $\xi_0\in\Gamma(\mathrm{\Sigma}_0 \mathrm{M}_1)$ a parallel spinor on $M_1$ and $\varphi_r\in\Gamma(\mathrm{\Sigma}_r \mathrm{M}_2)$ an anti-holomorphic K\"ahlerian twistor spinor on $M_2$ ($D^-\varphi_r=0$). Similarly $\psi_r=\xi_r\otimes\varphi_0$, in particular with $\varphi_0\in\Gamma(\mathrm{\Sigma}_0 \mathrm{M}_2)$ a parallel spinor on $M_2$, but as $M_2$ is not Ricci-flat, this term must vanish.

\textbf{II.} If $r$ is strictly larger than $m$ and $n$, then by applying the real or quaternionic structure $\mathfrak{j}$ to a K\"ahlerian twistor spinor in $\s{r}$ we get one in $\s{m+n-r}$, thus reducing to the first case. It then follows that a K\"ahlerian twistor spinor $\psi\in\Gamma(\s{r})$ is of the form $\psi=\xi_m\otimes\varphi_{r-m}$ with $\xi_m\in\Gamma(\mathrm{\Sigma}_m \mathrm{M}_1)$ a parallel spinor on $M_1$ and $\varphi_{r-m}\in\Gamma(\mathrm{\Sigma}_{r-m} \mathrm{M}_2)$ a holomorphic K\"ahlerian twistor spinor on $M_2$.

\textbf{III.} Let $r$ be a number between $m$ and $n$ and suppose that $m\leq r\leq n$. Since $\mathrm{\Sigma}_k \mathrm{M}_1$ exist only for $0\leq k\leq m$, then automatically $\psi_{m+1}=\cdots=\psi_{r}=0$. Integrating \eqref{d1d2} over $M$ we get as above $D_1^+\psi_k=D_2^-\psi_k=0$, $\forall k\leq m-1$ and $D_1^-\psi_k=D_2^+\psi_k=0$, $\forall k\geq 1$. From \eqref{twist1} and \eqref{twist2} it follows that $\psi_1, \ldots, \psi_{m-1}$ are parallel spinors in $\sigm$ and thus must vanish. The first component $\psi_0$ has as before the form $\psi_0=\xi_0\otimes\varphi_r$ with $\xi_0\in\Gamma(\mathrm{\Sigma}_0 \mathrm{M}_1)$ a parallel spinor on $M_1$ and $\varphi_r\in\Gamma(\mathrm{\Sigma}_r \mathrm{M}_2)$ an anti-holomorphic K\"ahlerian twistor spinor on $M_2$ and the last component is of the form $\psi_m=\xi_m\otimes\varphi_{r-m}$ with $\xi_m\in\Gamma(\mathrm{\Sigma}_m \mathrm{M}_1)$ a parallel spinor on $M_1$ and $\varphi_{r-m}\in\Gamma(\mathrm{\Sigma}_{r-m} \mathrm{M}_2)$ a holomorphic K\"ahlerian twistor spinor on $M_2$. 

The last possible case is when $n\leq r\leq m$. The same argument as above holds with $M_1$ and $M_2$ interchanged. As $M_2$ is assumed not to be Ricci-flat, then it carries no parallel spinors, showing that there are no nontrivial K\"ahlerian twistor spinors in this case.\qed
\end{proof}

\begin{Remark}
\normalfont{
From Theorem~\ref{product} it follows in particular that on a product of two compact spin K\"ahler manifolds any K\"ahlerian twistor spinor is a special K\"ahlerian twistor spinor. 
Moreover, since one of the factors must be a Ricci-flat manifold, it follows that the second factor, which in turn carries special K\"ahlerian twistor spinors, is an irreducible K\"ahler manifold with holonomy $U(m)$ (from Berger's list, where we eliminate the case of symmetric manifolds, which are in particular K\"ahler-Einstein and thus studied in Theorem~\ref{classke}).

If $n<m$, where $m$ is the complex dimension of the Ricci-flat factor $M_1$ and $n$ the complex dimension of the other factor $M_2$, then Theorem \ref{product} implies that there are no nontrivial K\"ahlerian twistor spinors in $\s{r}$ for $n<r<m$.}
\end{Remark}

\subsection{The Geometric Description}
In this section we give the main result, which is now a consequence of Theorems~\ref{2eigenval} and ~\ref{product} and the following splitting result proven by V.~Apostolov, T.~Dr\u aghici and A.~Moroianu \cite{adm}:

\begin{Theorem}\label{split}
Let $(M,g,J)$ be a compact K\"ahler manifold whose Ricci tensor has two distinct constant non-negative eigenvalues $\lambda$ and $\mu$. Then the universal cover of $(M,g,J)$ is the product of two simply connected K\"ahler-Einstein manifolds with Einstein constants $\lambda$ and $\mu$, respectively.
\end{Theorem}

\begin{Theorem}\label{class}
Let $(M^{2m},g,J)$ be a compact simply connected spin K\"ahler manifold of constant scalar curvature admitting nontrivial non-parallel K\"ahlerian twistor spinors in $\s{r}$ for an $r$ with $0<r<m$. Then $M$ is the product of a Ricci-flat manifold $M_1$ and an irreducible K\"ahler-Einstein manifold $M_2$, which must be one of the manifolds described in Theorem~\ref{classke}. More precisely, there exist anti-holomorphic (holomorphic) K\"ahlerian twistor spinors in at most one such $\s{r}$ with $r<\frac{m}{2}$ ($r>\frac{m}{2}$) and they are of the form:
\begin{equation}\label{tenstwist}
\psi=\xi_0\otimes\varphi_r \quad (\psi=\xi_{2r-m-1}\otimes\varphi_{m-r+1}),
\end{equation}
where $\xi_0\in\Gamma(\mathrm{\Sigma}_0 \mathrm{M}_1)$ ($\xi_{2r-m-1}\in\Gamma(\mathrm{\Sigma}_{2r-m-1} \mathrm{M}_1)$) is a parallel spinor and  $\varphi_r\in\Gamma(\mathrm{\Sigma}_{r} \mathrm{M}_2)$ ($\varphi_{m-r+1}\in\Gamma(\mathrm{\Sigma}_{m-r+1} \mathrm{M}_2)$) is an anti-holomorphic (holomorphic) K\"ahlerian twistor spinor. In particular, the complex dimension of the K\"ahler-Einstein manifold $M_2$ is $2r+1$ (resp. $2(m-r)+1$).
\end{Theorem}

\begin{proof}
Let $(M,g,J)$ be a K\"ahler manifold as in the hypothesis of the theorem and $\varphi\in\Gamma(\s{r})$ a K\"ahlerian twistor spinor. By Proposition~\ref{allsemi} and Remark~\ref{pos}, $\varphi$ is a special K\"ahlerian twistor spinor and, as usually, we may suppose that it is an anti-holomorphic K\"ahlerian twistor spinor. Then, by Theorem~\ref{2eigenval} the Ricci tensor has two constant non-negative eigenvalues: $\frac{S}{2(2r+1)}$ with multiplicity $2(2r+1)$ and $0$ with multiplicity $2(m-2r-1)$. From Theorem~\ref{split}, as $M$ is supposed to be simply connected, it follows that $M$ is the product of a Ricci-flat manifold $M_1$ and a K\"ahler-Einstein manifold $M_2$ of scalar curvature equal to $\frac{S}{2(2r+1)}$. By Theorem~\ref{product}, $\psi$ is of the form \eqref{tenstwist} with $\xi_0$ is a parallel spinor in $\Sigma_0 M_1$ and $\varphi_r$ is an anti-holomorphic K\"ahlerian twistor spinor in $\Sigma_{r} M_2$. We then conclude by applying Theorem~\ref{classke}. We notice that this result together with Corollary \ref{cordim1} also provides the dimension of the space of K\"ahlerian twistor spinors. \qed
\end{proof}

\begin{Remark}
\normalfont{
This result can be seen as a generalization of the geometric  description of limiting even dimensional K\"ahler manifolds for Kirchberg's inequality \eqref{kircheven} using the characterization in Theorem~\ref{characteven}. Thus, if $M$ is a limiting K\"ahler manifold of even complex dimension $m=2l$, then it admits an anti-holomorphic K\"ahlerian twistor spinor in $\s{l-1}$ (or equivalently a holomorphic K\"ahlerian twistor spinor in $\s{l+1}$). By Theorem~\ref{class}, $M$ is then the product of a $2$-dimensional flat manifold $M_1$ and a $(4l-2)$-dimensional K\"ahler-Einstein manifold $M_2$, which is a limiting manifold for Kirchberg's inequality \eqref{kirchodd} in odd dimensions.}
\end{Remark}

\begin{Remark}
\normalfont{If in Theorem \ref{class} the manifold $(M,g)$ is not assumed to be simply connected, then its universal Riemannian cover $(\tilde{M},\tilde{g})$ carries a unique spin structure and since its Ricci curvature is non-negative (as proven in Theorem~\ref{2eigenval}), it follows by a result of J. Cheeger and D. Gromoll (\cite[Theorem 6.65]{besse}) that $(\tilde{M},\tilde{g})$ is isometric to a Riemannian product $(\bar{M}\times\mathbb{R}^q,\bar{g}\times g_0)$, where $g_0$ is the canonical flat metric on $\mathbb{R}^q$ and $(\bar{M},\bar{g})$ is a compact simply connected manifold with positive Ricci curvature. In order to complete the classification of K\"ahler spin manifolds admitting nontrivial non-extremal K\"ahlerian twistor spinors, one has to analyze the existence of such spinors on the product $(\bar{M}\times\mathbb{R}^q,\bar{g}\times g_0)$ and the action of the fundamental group of $M$ on $\tilde{M}$. In the special case of limiting manifolds for the even dimensional Kirchberg inequality, this classification has been obtained by A. Moroianu, \emph{cf.} Theorem \ref{classeven}.}
\end{Remark}

In particular, Theorem~\ref{class} together with Proposition~\ref{eigentyper} answer a question raised by \mbox{K.-D.~Kirchberg} \cite{kirch2} about the description of all compact K\"ahler spin manifolds, whose square of the Dirac operator has the smallest eigenvalue of type $r$.

\subsection{Weakly Bochner Flat Manifolds}\label{sectwbf}

All the examples that we know of K\"ahler spin manifolds admitting special K\"ahlerian twistor spinors have parallel Ricci form, being thus in particular weakly Bochner flat. The purpose of this section is to show conversely, that any spin weakly Bochner flat manifold admitting special K\"ahlerian twistor spinors must have constant scalar curvature and thus, is described in Theorem~\ref{class}.

We first recall that a K\"ahler manifold $(M,g,J)$ is called {\it weakly Bochner flat} if its Bochner tensor (which is defined as the projection of the Weyl tensor onto the space of K\"ahlerian curvature tensors) is co-closed. In \cite{acg} (cf. Proposition~1), the codifferential of the Bochner tensor is computed using the Matsushima identity and it is proven that a K\"ahler manifold is weakly Bochner flat if and only if the normalized Ricci form defined by
\begin{equation}\label{rhonorm}
\tilde\rho:=\rho-\frac{1}{2(m+1)}S\Omega
\end{equation}
is a Hamiltonian $2$-form, \emph{i.e.} it satisfies the following equation 
\begin{equation}\label{wbf}
\nabla_{X}\tilde\rho=\frac{1}{4(m+1)}(dS\wedge JX-d^cS\wedge X),
\end{equation}
for all vector fields $X$.

\begin{Proposition}
Let $(M,g,J)$ be a spin weakly Bochner flat manifold and $\varphi\in\Gamma(\s{r})$ (with $0<r<m$) be a nontrivial anti-holomorphic (or holomorphic) K\"ahlerian twistor spinor. Then the scalar curvature $S$ of the metric $g$ is constant.
\end{Proposition}
\begin{proof}
Let $\varphi\in\Gamma(\s{r})$ (with $0<r<m$, $r\neq\frac{m}{2}$) be a nontrivial anti-holomorphic K\"ahlerian twistor spinor (using the isomorphism $\mathfrak{j}$ the same argument holds for a holomorphic K\"ahlerian twistor spinor). First we notice that using the projections onto $T^{(1,0)}M$ and $T^{(0,1)}M$, the equation \eqref{wbf} is equivalent to the following equations:
\begin{equation}\label{x+}
i\nabla_{X^+}\tilde\rho=\frac{1}{2(m+1)}X^+\wedge(dS)^-,
\end{equation}
\begin{equation}\label{x-}
i\nabla_{X^-}\tilde\rho=-\frac{1}{2(m+1)}X^-\wedge(dS)^+.
\end{equation}
From \eqref{rhonorm} and \eqref{x+} we obtain:
\begin{equation*}
i\nabla_{X^+}\rho=i\nabla_{X^+}\tilde\rho+\frac{1}{2(m+1)}iX^+(S)\Omega=\frac{1}{2(m+1)}[X^+\wedge(dS)^-+iX^+(S)\Omega].
\end{equation*}
Applying this equation to $\varphi$ and using \eqref{ds-} we get
\begin{equation}\label{lnab-rho1}
i\nabla_{X^+}\rho\cdot\varphi=\frac{1}{2(m+1)}(X^+\wedge(dS)^-)\cdot\varphi+\frac{1}{2(m+1)}iX^+(S)\Omega\cdot\varphi=\frac{m-2r+1}{2(m+1)}X^+(S)\varphi.
\end{equation}
On the other hand, by \eqref{lnab-rho} we have
\begin{equation}\label{lnab-rho2}
i\nabla_{X^+}\rho\cdot\varphi=\frac{1}{2(2r+1)}X^+(S)\varphi.
\end{equation}
Comparing the equations \eqref{lnab-rho1} and \eqref{lnab-rho2} we get
\[\frac{r(m-2r)}{2(m+1)(2r+1)}X^+(S)\varphi=0.\]
As $r\neq \frac{m}{2}$ and $r\neq 0$, it follows that $X^+(S)=0$ at all points where $\varphi$ does not vanish, thus proving that $S$ must be constant.\qed
\end{proof}


\begin{thebibliography}{99}

\bibitem{acg} \textsc{V. Apostolov, D.M.J.Calderbank, P.Gauduchon,} {\sl Hamiltonian 2-forms in K\"ahler Geometry I: General Theory,} J. Diff. Geom. {\bf 73} (2006), 359--412.

\bibitem{adm} \textsc{V. Apostolov, T. Dr\u aghici, A. Moroianu,} {\sl A splitting Theorem for K\"ahler Manifolds with Constant Eigenvalues of the Ricci Tensor}, Internat. J. Math. {\bf 12} (2001), 769--789.

\bibitem{teub} \textsc{H. Baum, Th. Friedrich, R. Grunewald, I. Kath,} {\sl Twistor and Killing Spinors on Riemannian Manifolds}, Teubner Verlag Leipzig/Stuttgart(1991).

\bibitem{baer} \textsc{Ch. B\"ar,} {\sl Real Killing Spinors and Holonomy}, Commun. Math. Phys. {\bf 154} (1993), 509--521. 

\bibitem{besse} \textsc{A. L. Besse,} {\sl Einstein Manifolds}, Ergebnisse der Mathematik (3), Springer Verlag, Berlin (1987). 

\bibitem{branson} \textsc{Th. Branson,} {\sl Stein-Weiss Operators and Ellipticity}, J. Funct. Anal. {\bf 151} (1997), 334--383. 

\bibitem{fr1} \textsc{Th. Friedrich,} {\sl Der erste Eigenwert des Dirac-Operators einer kompakten Riemannschen Mannigfaltigkeit nichtnegativer Skalarkr\"ummung}, Math. Nach. {\bf 97} (1980), 117--146. 

\bibitem{fr2} \textsc{Th. Friedrich,} {\sl On the Conformal Relation between Twistor and Killing Spinors}, Suppl. Rend. Circ. Mat. Palermo (1989), 59--75. 

\bibitem{fr} \textsc{Th. Friedrich,} {\sl Dirac Operatoren in der Riemannschen Geometrie}, Advanced Lectures in Mathematics, Vieweg Braunschweig (1997).

\bibitem{pg} \textsc{P. Gauduchon,}  {\sl L'Op{\'e}rateur de Penrose k\"ahl{\'e}rien et les in{\'e}galit{\'e}s de Kirchberg}, preprint, 1993.

\bibitem{hijth} \textsc{O. Hijazi,}  {\sl Op{\'e}rateurs de Dirac sur les vari{\'e}t{\'e}s riemanniennes : minoration des valeurs propres,} Th{\`e}se de Doctorat, {\'E}cole Polytechnique-Paris VI (1984).

\bibitem{hij} \textsc{O. Hijazi,}  {\sl Eigenvalues of the Dirac Operator on Compact K\"ahler Manifolds}, Comm. Math. Phys. {\bf 160} (1994), 563--579.

\bibitem{hitchin} \textsc{N. Hitchin,}  {\sl Harmonic Spinors}, Adv. Math.{\bf 14} (1974), 1--55.

\bibitem{kirch86} \textsc{K.-D. Kirchberg,} {\sl An Estimation for the First Eigenvalue of the Dirac Operator on Closed K\"ahler Manifolds of Positive Scalar Curvature}, Ann. Global Anal. Geom. {\bf 3}  (1986), 291--325. 

\bibitem{kirch2} \textsc{K.-D. Kirchberg,} {\sl The First Eigenvalue of the Dirac Operator on K\"ahler Manifolds}, J. Geom. Phys. {\bf 7} (1990), 449--468.

\bibitem{kirch1} \textsc{K.-D. Kirchberg,} {\sl Properties of K\"ahlerian Twistor Spinors and Vanishing Theorems}, Math. Ann. {\bf  293} no. 2 (1992),   349--369. 

\bibitem{kirch} \textsc{K.-D. Kirchberg,} {\sl Killing Spinors on K\"ahler Manifolds}, Ann. Glob. Anal. Geom. {\bf 11} (1993), 141--164.

\bibitem{ku} \textsc{K.-D. Kirchberg, U. Semmelmann,} {\sl Complex Contact Structures and the First Eigenvalue of the Dirac Operator
on K\"ahler Manifolds}, Geom. and Funct. Analysis {\bf 5} (1995), 604--618.

\bibitem{spin} \textsc{ H. B. Lawson, M.-L. Michelson} {\sl Spin Geometry}, Princeton Mathematical Series, 38. Princeton University Press (1989).

\bibitem{lich} \textsc{A. Lichnerowicz,} {\sl La premi{\`e}re valeur propre de l'op{\'e}rateur de Dirac pour une vari{\'e}t{\'e} k\"ahleri{\'e}nne et son cas limite}, C.R. Acad. Sci. Paris, t. {\bf 311}, Serie {\bf I} (1990), 717--722. 

\bibitem{lichn1} \textsc{A. Lichnerowicz,} {\sl Spineurs harmoniques et spineurs-twisteurs en g{\'e}om{\'e}trie k\"ahl{\'e}rienne et conform{\'e}ment k\"ahleri{\'e}nne}, C.R. Acad. Sci. Paris, t. {\bf 311}, Serie {\bf I} (1990), 883--887. 

\bibitem{am_odd} \textsc{A. Moroianu,} {\sl La premi{\`e}re valeur propre de l'op{\'e}rateur de Dirac sur les vari{\'e}t{\'e}s k\"ahl{\'e}riennes compactes}, Commun. Math. Phys. {\bf 169} (1995), 373--384.

\bibitem{am_th} \textsc{A. Moroianu,} {\sl Op{\'e}rateur de Dirac et submersions riemanniennes,} Th{\`e}se de Doctorat, {\'E}cole Polytechnique (1996).

\bibitem{am_even1} \textsc{A. Moroianu,} {\sl On Kirchberg's Inequality for Compact K{\"a}hler Manifolds of Even Complex Dimension}, Ann. Global Anal. Geom. {\bf 15} (1997), 235--242.

\bibitem{am_even} \textsc{A. Moroianu,} {\sl K\"ahler Manifolds with Small Eigenvalues of the Dirac Operator and a Conjecture of Lichnerowicz}, Ann. Inst. Fourier. {\bf 49} (1999), 1637--1659.

\bibitem{am_lectures} \textsc{A. Moroianu,} {\sl Lectures on K{\"a}hler Geometry}, London Mathematical Society Student Text {\bf 69}, Cambridge University Press, Cambridge, 2007.

\bibitem{ms1} \textsc{A. Moroianu, U. Semmelmann,} {\sl Twistor Forms on K\"ahler Manifolds}, Ann. Scuola Norm. Sup. Pisa {\bf II} (2003), 823--845. 

\bibitem{ms2} \textsc{A. Moroianu, U. Semmelmann,} {\sl Twistor Forms on Riemannian Products}, J. Geom. Phys. {\bf 58} (2008), 1343--1345.

\bibitem{ss} \textsc{S. Seifarth, U. Semmelmann,} {\sl The Spectrum of the Dirac Operator on Complex Projective Spaces}, SFB 288 preprint, no. {\bf 95}, Berlin (1993).

\bibitem{uwethesis} \textsc{U. Semmelmann,} {\sl K\"ahlersche Killingspinoren und komplexe Kontaktstrukturen,}, Thesis, Humboldt Universität zu Berlin (1995).

\bibitem{uweshort} \textsc{U. Semmelmann,} {\sl A Short Proof of Eigenvalue Estimates for the Dirac Operator on Riemannian and K\"ahler Manifolds}, Diff. Geom. and Appl., Proceedings, Brno (1998), 137--140.

\bibitem{uwehabil} \textsc{U. Semmelmann,} {\sl Conformal Killing Forms on Riemannian Manifolds}, Habilitationsschrift (2001).

\end{thebibliography}
\end{document}